\documentclass[11pt]{amsart}
\usepackage[top=1in, bottom=1in, left=1in, right=1in]{geometry}

\usepackage{mathtools,mathrsfs,amsthm,amsfonts,amssymb,graphicx,appendix,cancel}
\usepackage{xspace,enumerate,comment,bbm,nth}

\usepackage{tikz,pgfplots}

\usepackage{hyperref}
\usepackage[normalem]{ulem}   

\usepackage[active]{srcltx}

\numberwithin{equation}{section}
\newtheorem{theorem}{\sc Theorem}[section]

\newtheorem{definition}[theorem]{\sc Definition}
\newtheorem{lemma}[theorem]{\sc Lemma}

\theoremstyle{plain}

\newtheorem{prop}[theorem]{\sc Proposition}
\newtheorem{cor}[theorem]{\sc Corollary}

\theoremstyle{remark}
\newtheorem{remark}[theorem]{\sc Remark}

 \newcommand{\ol}[1]{\overline{#1}}
 \newcommand{\ul}[1]{\underline{#1}}

\newcommand{\F}{\mathcal F}
\newcommand{\R}{\mathbb{R}}
\newcommand{\Sol}{\mathcal S}
\newcommand{\Solfrak}{\mathfrak{S}}
\newcommand{\superSol}{\overline{\mathfrak{S}}}

\newcommand{\N}{\mathbb{N}}

\renewcommand{\P}{\mathbb{P}}
\newcommand{\parts}{\mathscr{P}}
\newcommand{\EE}{\mathbb{E}}

\newcommand{\cal}[1]{\mathcal #1}

\newcommand{\HV}{{\mathcal H}}
\newcommand{\HF}{{\mathcal H}}

\newcommand{\eps}{\varepsilon}
\renewcommand{\epsilon}{\varepsilon}

\newcommand{\landa}{\lambda}
\renewcommand{\emptyset}{\varnothing}

\renewcommand{\ge}{\geqslant}

\newcommand{\CC}{\D C}

\newcommand{\ccyl}{[0,+\infty)\times\R}

\newcommand{\Ham}{\mathscr{H}(\alpha_0,\alpha_1,\gamma)}
\newcommand{\Hamall}{\mathscr{H}}

\newcommand{\D}[1]{\mbox{\rm #1}}

\begin{document}
	
\title[Stochastic homogenization of viscous HJ
equations in 1d]{Stochastic homogenization\\ of  nondegenerate viscous HJ equations in 1d}

\author[A.\ Davini]{Andrea Davini}
\address{Andrea Davini\\ Dipartimento di Matematica\\ {Sapienza} Universit\`a di
  Roma\\ P.le Aldo Moro 2, 00185 Roma\\ Italy}
\email{davini@mat.uniroma1.it}

\date{February 26, 2024}

\subjclass[2010]{35B27, 35F21, 60G10.} 
\keywords{Viscous Hamilton-Jacobi equation, stochastic homogenization, stationary ergodic random environment, sublinear corrector}

\begin{abstract}
We prove homogenization for a nondegenerate viscous Hamilton-Jacobi equation in dimension one in stationary ergodic environments
with a superlinear (nonconvex) Hamiltonian of fairly general type.
\end{abstract}

\maketitle

\section{Introduction}\label{sec:intro}

In this paper we are concerned with the asymptotic behavior,  as $\epsilon\to 0^+$, of solutions of a  viscous
Hamilton-Jacobi (HJ) equation of the form
\begin{equation}\label{eq:introHJ}
  \partial_t u^\epsilon=\epsilon a\left({x}/{\epsilon},\omega\right) \partial^2_{xx} u ^\epsilon+H\left({x}/{\eps},\partial_x u^\epsilon,\omega\right) \qquad \hbox{in $(0,+\infty)\times\R$,}
  \tag{EHJ$_\eps$}
\end{equation}
where $H:\R\times\R\times\Omega\to\R$ is a  superlinear (in the momentum) Hamiltonian of fairly general type, and $a:\R\times\Omega\to [0,1]$ is Lipschitz on $\R$ for every fixed $\omega$. The dependence of the equation on the random environment $(\Omega,\F,\P)$ enters through the Hamiltonian $H(x,p,\omega)$ and the diffusion coefficients $a(x,\omega)$, which are assumed to be stationary with respect to shifts in $x$, and bounded and Lipschitz continuous on $\R$, for every fixed $p\in\R$ and $\omega\in\Omega$.  The diffusion coefficient is assumed nondegenerate, i.e.,  $a(\cdot,\omega)>0$ in $\R$ almost surely.  Under these hypotheses, we prove homogenization for equation \eqref{eq:introHJ}, see Theorem \ref{thm:genhom} 
for  the precise statement. The full set of assumptions is presented in Section \ref{sec:result}, here we want to stress that we do not require any convexity condition of any kind on the Hamiltonian in the momentum.

This is the first work where the homogenization of nondegenerate viscous HJ equations in 1d in random media is proved in full generality, at least as far as superlinear Hamiltonians are concerned, thus extending to this setting the results established in the 1d inviscid case in \cite{ATY_nonconvex, Gao16}.   
With respect to previous contributions on the subject, the crucial step forward of our work consists in dealing with Hamiltonians  where the dependence on $x$ and $p$ is not necessarily decoupled and, more important, in getting rid of any additional requirement on the pair $(a,H)$, notably the hill and valley conditions that were assumed in \cite{KYZ20,  DK17, DKY23}. 

We would also  like to point out that our proof is  technical simple. One could compare it with the ones available on the topic of stationary ergodic homogenization of HJ equations in 1d, both in the viscous and in the inviscid case. 
It  relies, among other things, on the use of a novel idea which is presented in Theorem \ref{teo lower bound}.
\footnote{This idea was first introduced by the author in \cite[Theorem 4.5]{D23c}. Unfortunately, the argument to prove the main homogenization result therein contained is also based on \cite[Theorem 4.3]{D23c}, whose proof contains a non-amendable flaw.} This allows us to prove, in the end, that the set of $\theta$ in $\R$ for which the related ``cell'' problem associated with \eqref{eq:introHJ} does not admit correctors is the possibly countable union of pairwise disjoint, bounded open intervals which correspond to flat parts of the effective Hamiltonian. This complements, in our setting, the information about the existence of correctors provided by \cite{CaSo17} for possibly degenerate viscous HJ equations of the form \eqref{vHJ}  in any space dimension. We refer the reader to Section \ref{sec:outline} for a more detailed description of our proof strategy.  

The present work, in pair with \cite{D23b}, implies in particular homogenization of \eqref{eq:introHJ} when the Hamiltonian is additionally assumed quasiconvex in $p$ and the diffusion coefficient is possibly degenerate (but {non null}). Indeed, in \cite{D23b} we have established homogenization for quasiconvex Hamiltonians belonging to the same class herein considered in the degenerate regime, i.e., when the diffusion coefficient satisfies $\min_{\R} a(\cdot,\omega)=0$ almost surely and $a\not\equiv 0$. The effective Hamiltonian is also shown to inherit the quasiconvexity from $H$. This is not to be expected in general in the nondegenerate case, in fact it does not even hold in the classical periodic case, as it was recently pointed out in \cite{KY23}.

\subsection{Literature review.}
Equations of the form \eqref{eq:introHJ} are a subclass of general
viscous stochastic HJ equations
\begin{equation}\label{vHJ}
  \partial_tu^\epsilon= \epsilon\text{tr}\left(A\left(\frac{x}{\epsilon},\omega\right) D^2 u^\epsilon\right)+H\left(D u^\epsilon,\frac{x}{\epsilon},\omega\right)\qquad \hbox{in $(0,+\infty)\times\R^d$,}
\end{equation}
where $A(\cdot,\omega)$ is a bounded, symmetric, and nonnegative
definite $d\times d$ matrix with a Lipschitz square root. The ingredients of the equation are assumed to be stationary with respect to shifts in $x$. 
This setting encompasses the periodic and the quasiperiodic cases, for which homogenization has been proved under fairly general assumptions by showing the existence of (exact or approximate) correctors, i.e., sublinear functions that solve an associated stationary HJ equations \cite{LPV, Evans92, I_almostperiodic, LS_almostperiodic}. 
In the stationary ergodic setting, such solutions do not exist in general, as it was shown in \cite{LS_correctors}, see also \cite{DS09,CaSo17} for inherent discussions and results. This is the main reason why the extension of the homogenization theory to random media is nontrivial 
and required the development of new arguments. 

The first homogenization results in this framework were obtained for convex Hamiltonians in the case of inviscid equations in \cite{Sou99, RT} and then for their viscous counterparts in \cite{LS2005,KRV}. 
By exploiting the metric character of first order HJ equations, homogenization has been extended to the case of quasiconvex Hamiltonians, first in dimension 1 \cite{DS09} and then in any space dimension \cite{AS}. 

The topic of homogenization in stationary ergodic media for HJ equations that are nonconvex in the gradient variable
remained an open problem for about fifteen years. Recently, it has been shown via counterexamples, first in the inviscid case \cite{Zil,FS}, then in the viscous one \cite{FFZ}, that homogenization can fail for Hamiltonians of the form $H(x,p,\omega):=G(p)+V(x,\omega)$ whenever $G$ has a strict saddle point. This has shut the door to the possibility of having a general qualitative homogenization theory  in the stationary ergodic setting in dimension $d\geqslant 2$, at least without imposing further mixing conditions on the stochastic environment. 

On the positive side, a quite general homogenization result, which includes as particular instances both the inviscid and viscous cases, has been established in \cite{AT} for Hamiltonians that are positively homogeneous of degree $\alpha\geqslant 1$ and under a finite range of dependence condition on the probability space $(\Omega,\F, \P)$. 

In the inviscid case, homogenization has been proved in \cite{ATY_1d,Gao16} in dimension $d=1$ for a rather general class of coercive and nonconvex Hamiltonians, and in any space dimension for Hamiltonians of the form $H(x,p,\omega)=\big(|p|^2-1\big)^2+V(x,\omega)$, see   \cite{ATY_nonconvex}.\smallskip 

Even though the addition of a diffusive term is not expected to prevent homogenization, the literature on viscous HJ equations 
in stationary ergodic media has remained rather limited until very recently. 
The viscous case is in fact known to present additional 
challenges which cannot be overcome by mere
modifications of the methods used for $a\equiv 0$.

Apart from already mentioned work \cite{AC18}, 
several progresses concerning the homogenization of the viscous HJ equation \eqref{vHJ} with
nonconvex Hamiltonians have been recently made in 
\cite{DK17, YZ19, KYZ20, Y21b, DKY23, D23a, D23b}. In the joint paper \cite{DK17}, we have
shown homogenization of \eqref{vHJ} with $H(x,p,\omega)$ which
are ``pinned'' at one or several points on the $p$-axis and convex in
each interval in between. For example, for every $\alpha>1$ the
Hamiltonian $H(p,x,\omega)=|p|^\alpha-c(x,\omega)|p|$ 
is pinned at $p=0$ (i.e.\
$H(0,x,\omega)\equiv
\mathrm{const}$) and convex in $p$ on each of the two intervals $(-\infty, 0)$ and
$(0,+\infty)$.
Clearly, adding a non-constant potential $V(x,\omega)$ breaks the pinning
property. In particular, homogenization of equation \eqref{vHJ} for $d=1$, $A\equiv \mathrm{const}>0$ and 
 $H(p,x,\omega):=\frac12\,|p|^2-c(x,\omega)|p|+V(x,\omega)$ with $c(x,\omega)$ bounded and strictly positive 
remained an open problem even when $c(x,\omega)\equiv c>0$.  
Homogenization for this kind of equations with $A\equiv 1/2$ and $H$ as
above with $c(x,\omega)\equiv c>0$ was proved in \cite{KYZ20} 
under a novel hill and valley condition,\footnote{Such a hill and valley condition is fulfilled for a wide class of typical random environments without any restriction on their mixing properties, see  \cite[Example 1.3]{KYZ20} and \cite[Example 1.3]{YZ19}.  It is however not satisfied if the potential is ``rigid'', for example, in the periodic case.}
 that was introduced in \cite{YZ19} 
to study a sort of discrete version of this problem. 
The approach of \cite{KYZ20} relies on the Hopf-Cole
transformation, stochastic control representations of solutions and
the Feynman-Kac formula. It is applicable to \eqref{eq:introHJ}
with $H(x,p,\omega):=G(p)+V(x,\omega)$ and $G(p):=\frac12|p|^2-c|p|=\min\{\frac12p^2-cp,\frac12p^2+cp\}$ only. 
In the joint paper 
\cite{DK22}, we have proposed a different proof solely based on PDE methods. 
This new approach is flexible enough to be applied to the possible degenerate case 
$a(x,\omega)\ge 0$, and to any $G$ which is a minimum of a finite number of convex
superlinear functions $G_i$  having the same
minimum. The original hill and valley condition assumed in \cite{YZ19,KYZ20} is weakened in 
favor of a scaled hill and valley condition.
The arguments crucially rely on this condition and on the fact that all the functions $G_i$ have the same minimum.


The PDE approach introduced in \cite{DK22} was subsequently refined in \cite{Y21b} in order to prove 
homogenization for equation \eqref{eq:introHJ} when the function $G$ is superlinear and quasiconvex,  $a>0$ and the pair 
$(a,V)$ satisfies a scaled hill condition equivalent to the one adopted in \cite{DK22}. 
The core of the proof consists in showing existence and uniqueness of correctors  that possess stationary derivatives satisfying suitable bounds. These kind of results were obtained in \cite{DK22} by proving tailored-made  comparison principles and by exploiting a general result from \cite{CaSo17}  (and this was the only point where the piecewise convexity of $G$ was used). The novelty brought in by \cite{Y21b} 
relies on the nice observation that this can be directly proved via ODE arguments, which are viable since we are in one space dimension and $a>0$.

By reinterpreting this approach in the viscosity sense and by making a more substantial use of viscosity techniques, the author has extended in \cite{D23a} to the possible degenerate case $a\geqslant 0$ the homogenization result established in \cite{Y21b}, providing in particular a unified proof which encompasses both the inviscid and the viscous case.  


A substantial improvement of the results of \cite{DK22}  is provided in our joint work 
 \cite{DKY23}, where the  
main novelty consists in allowing a  general
superlinear $G$ without any restriction on its shape or the
number of its local extrema. Our analysis crucially relies on the assumption that the pair 
$(a,V)$ satisfies the scaled hill and valley condition and, differently from \cite{DK22}, 
that the diffusive coefficient $a$ is nondegenerate, i.e., $a>0$. 
%
The proof consists in showing existence of suitable correctors whose derivatives 
are confined on different branches of $G$, as well as on a strong induction argument 
which uses as  base case the homogenization result established in \cite{Y21b}.

Our recent work \cite{D23b} is the first one where a homogenization result for equation \eqref{eq:introHJ} has been established without imposing any hill or valley-type condition of any sort on the pair $(a, H)$. In \cite{D23b} we prove homogenization of equation  \eqref{eq:introHJ} for a quasiconvex Hamiltonian belonging to the same class considered in the present paper and with the only further assumption that the diffusion coefficient $a$ is not identically zero and vanishes at some points or on some regions of $\R$, almost surely. This assumption on the diffusion coefficient is crucially exploited to obtain suitable existence, uniqueness and regularity results of correctors, as well as a quasi-convexity property of the expected effective Hamiltonian.
This is not to be expected in general in the nondegenerate case. In fact, it does not even hold in the classical periodic case, as it was recently pointed out in \cite{KY23}, where the authors identify a class of quasiconvex functions $G(p)$ and 1-periodic potentials $V(x)$ for which the effective Hamiltonian arising in the homogenization of equation \eqref{eq:introHJ}  with $H(x,p):=G(p)+V(x)$ and $a\equiv 1$ is not quasiconvex. 

\subsection{Outline of the paper.} In Section \ref{sec:result} we present the setting and the standing assumptions and we state our homogenization result, see Theorem \ref{thm:genhom}. In Section \ref{sec:outline} we describe the strategy we will follow to prove homogenization. In Section \ref{sec:correctors} we characterize the set of admissible $\lambda$ for which the corrector equation can be solved and we correspondingly show existence of (deterministic) solutions, with $\omega$ treated as a fixed parameter.  The derivatives of such solutions can be put in one-to-one correspondence with solutions of a corresponding ODE. We exploit this duality to show the existence of random correctors with maximal and minimal slope, for any admissible $\lambda$. 
Section \ref{sec:bridging} is addressed to prove Theorem \ref{teo lower bound}, where we provide a novel argument to suitably bridge a pair of distinct stationary solution of the same ODE. In Section \ref{sec:homogenization}  we give the proof of our homogenization result. 
The paper ends with two appendices. In the first one, we have collected some results that we repeatedly use in the paper  concerning stationary sub and supersolutions of the ODEs herein considered. The second one contains some PDE results needed for our proofs.\medskip

\noindent{\textsc{Acknowledgements. $-$}} The author is a member of the INdAM Research Group GNAMPA.

\section{Preliminaries}\label{sec:preliminaries}
\subsection{Assumptions and statement of the main result}\label{sec:result}

We will denote by $\CC(\R)$ and $\CC(\R\times\R)$ the Polish spaces of continuous functions  on $\R$ and on $\R\times\R$, endowed with a metric inducing the topology of uniform convergence on compact subsets of $\R$ and of $\R\times\R$, respectively. 

The triple $(\Omega,\F, \P)$ denotes a probability space, where $\Omega$ is a Polish space, ${\cal F}$ is the $\sigma$-algebra of Borel subsets of $\Omega$, and $\P$ is a complete probability measure on $(\Omega,{\cal F})$.\footnote{The assumption that $\Omega$ is a Polish space and $\P$ is a complete probability measure is used in many points of the paper, but only to show joint measurability of the random objects therein introduce, see the proofs of Lemmas \ref{lemma lambda zero} and \ref{lemma measurable minimal sol}, of Theorem \ref{teo lower bound} and of Proposition \ref{prop D}.} We will denote by ${\cal B}$ the Borel $\sigma$-algebra on $\R$ and equip the product space $\R\times \Omega$ with the product $\sigma$-algebra ${\mathcal B}\otimes {\cal F}$.

We will assume that $\P$ is invariant under the action of a one-parameter group $(\tau_x)_{x\in\R}$ of transformations $\tau_x:\Omega\to\Omega$. More precisely, we assume that the mapping
$(x,\omega)\mapsto \tau_x\omega$ from $\R\times \Omega$ to $\Omega$ is measurable, $\tau_0=id$, $\tau_{x+y}=\tau_x\circ\tau_y$ for every $x,y\in\R$, and $\P\big(\tau_x (E)\big)=\P(E)$ for every $E\in{\cal F}$ and $x\in\R$. We will assume in addition that the action of $(\tau_x)_{x\in\R}$ is {\em ergodic}, i.e., any measurable function $\varphi:\Omega\to\R$ satisfying $\P(\varphi(\tau_x\omega) = \varphi(\omega)) = 1$ for every fixed $x\in{\R}$ is almost surely equal to a constant.
 If $\varphi\in L^1(\Omega)$, we write $\EE(\varphi)$ for the mean of $\varphi$ on $\Omega$, i.e. the quantity
$\int_\Omega \varphi(\omega)\, d \P(\omega)$.

A measurable function $f:\R\times \Omega\to \R$ is said to be {\em stationary} with respect to $(\tau_x)_{x\in\R}$ if  $f(x+y,\omega)=f(x,\tau_y\omega)$ for every $x,y\in\R$ and $\omega\in\Omega$.\smallskip 

In this paper, we will consider an equation of the form 
\begin{equation}\label{eq:generalHJ}
\partial_t u=a(x,\omega) \partial^2_{xx} u +H(x,\partial_x u,\omega),\quad\hbox{in}\ (0,+\infty)\times\R,
\end{equation}
where $a:\R\times\Omega\to [0,1]$ is a stationary function satisfying the following assumptions  
for some constant $\kappa > 0$:
\begin{itemize}
\item[(A1)] $a(\cdot,\omega)>0$ on $\R$ for $\P$-a.e. $\omega\in\Omega$;\smallskip
\item[(A2)]  $\sqrt{a(\,\cdot\,,\omega)}:\R\to [0,1]$\ is $\kappa$--Lipschitz continuous for all $\omega\in\Omega$.\footnote{
Note that (A2) implies that $a(\cdot,\omega)$ is $2\kappa$--Lipschitz in $\R$ for all $\omega\in\Omega$.  Indeed, 
for all $x,y\in\R$ we have 
\[
|a(x,\omega) - a(y,\omega)| = |\sqrt{a(x,\omega)} + \sqrt{a(y,\omega)}||\sqrt{a(x,\omega)} - \sqrt{a(y,\omega)}| \leqslant 2\kappa|x-y|.
\]
}
\end{itemize}
As for the Hamiltonian $H:\R\times\R\times\Omega\to\R$, we will assume that it is stationary with respect to shifts in $x$ variable, i.e., $H(x+y,p,\omega)=H(x,p,\tau_y\omega)$ for every $x,y\in\R$, $p\in\R$ and $\omega\in\Omega$, and that it belongs to the class  $\Ham$ defined as follows. 
\begin{definition}\label{def:Ham}
	A function $H:\R\times\R\times\Omega\to\R$ is said to be in the class $\Ham$ if it satisfies the following conditions, for some constants $\alpha_0,\alpha_1>0$ and $\gamma>1$:
	\begin{itemize}
		\item[(H1)] \quad $\alpha_0|p|^\gamma-1/\alpha_0\leqslant H(x,p,\omega)\leqslant\alpha_1(|p|^\gamma+1)$\quad for all ${(x,p,\omega)\in\R\times\R\times\Omega}$;\medskip
		\item[(H2)]\quad $|H(x,p,\omega)-H(x,q,\omega)|\leqslant\alpha_1\left(|p|+|q|+1\right)^{\gamma-1}|p-q|$\quad for all $(x,\omega)\in\R\times\Omega$ and $p,q\in\R$;\medskip
	\item[(H3)] \quad $|H(x,p,\omega)-H(y,p,\omega)|\leqslant\alpha_1(|p|^\gamma+1)|x-y|$\quad 
	for all $x,y,p\in\R$ and $(p,\omega)\in\R\times\Omega$.\medskip		
\end{itemize}
We will denote by $\Hamall$ the union of the families $\Ham$, where $\alpha_0,\alpha_1$ vary in $(0,+\infty)$ and $\gamma$ in $(1,+\infty)$. 
\end{definition}

Solutions, subsolutions and supersolutions of \eqref{eq:generalHJ}
will be always understood in the viscosity sense, see \cite{CaCa95, users,barles_book,bardi}, and implicitly assumed continuous, without any further specification. Assumptions (A2) and (G1)-(G3)
guarantee well-posedness in $\D{UC}(\ccyl)$ of the Cauchy problem for
the parabolic equation \eqref{eq:generalHJ}, as well as Lipschitz
estimates for the solutions under appropriate assumptions on the
initial condition, see Appendix \ref{app:PDE} for more
details. We stress that our results hold (with the same proofs) under any other set of assumptions apt to ensure the same kind of PDE results. 

The purpose of this paper is to prove the following homogenization result. 
\begin{theorem}\label{thm:genhom}
	Suppose $a$ satisfies (A1)-(A2) and $H\in\Hamall$. Then 
the viscous HJ equation \eqref{eq:introHJ} homogenizes, i.e., there exists a continuous and coercive function 
$\HV(H):\R\to\R$, called {\em effective Hamiltonian}, and a set $\hat\Omega$ of probability 1 such that, for every uniformly continuous function $g$ on $\R$ and every $\omega\in\hat\Omega$, the solutions $u^\epsilon(\cdot,\cdot,\omega)$ of \eqref{eq:introHJ} satisfying $u^\epsilon(0,\,\cdot\,,\omega) = g$ converge,
    locally uniformly on $[0, +\infty)\times \R$ as $\epsilon\to 0^+$, to the unique solution $\ol{u}$ of 
    \begin{eqnarray*}
   \begin{cases}
    \partial_t \ol{u} = \HV(H)(\partial_x\ol{u}) & \hbox{in $(0,+\infty)\times\R$}\\
    \ol{u}(0,\,\cdot\,) = g & \hbox{in $\R$}.
    \end{cases}
    \end{eqnarray*}
Furthermore, $\HV(H)$ is locally Lipschitz and superlinear. 
\end{theorem}

We remark that the effective Hamiltonian $\HV(H)$ also depends on the diffusion coefficient $a$. 
Since $a$ will remain fixed throughout the paper, we will not keep track of this in our notation. 

\subsection{Description of our proof strategy}\label{sec:outline}

In this section, we outline the strategy that we will follow to prove the homogenization results stated in Theorem \ref{thm:genhom}.\smallskip

Let us denote by $u_\theta(\,\cdot\,,\,\cdot\,,\omega)$ the unique Lipschitz solution to \eqref{eq:generalHJ} with initial condition $u_\theta(0,x,\omega)=\theta x$ on $\R$, and let us introduce the following deterministic quantities, defined almost surely on $\Omega$, see \cite[Proposition 3.1]{DK22}:
\begin{eqnarray}\label{eq:infsup}
	\HV^L(H) (\theta):=\liminf_{t\to +\infty}\ \frac{u_\theta(t,0,\omega)}{t}\quad\text{and}\quad 
	\HV^U(H) (\theta):=\limsup_{t\to +\infty}\ \frac{u_\theta(t,0,\omega)}{t}.
\end{eqnarray}
In view of \cite[Lemma 4.1]{DK17}, proving homogenization amounts to showing that $\HV^L(H)(\theta) = \HV^U(H)(\theta)$ for every $\theta\in\R$. If this occurs, their common value is denoted by $\HV(H)(\theta)$. The function $\HV(H):\R\to\R$ is called the effective Hamiltonian associated with $H$. It has already appeared in the statement of Theorem \ref{thm:genhom}.\smallskip

The we first remark that we can reduce to prove Theorem  \ref{thm:genhom} for stationary Hamiltonians $H$ belonging to $\Hamall(\alpha_0,\alpha_1,\gamma)$ for constants $\alpha_0,\, \alpha_1>0$ and $\gamma>2$.  In fact, the following holds. 

\begin{prop}\label{prop reduction}
If Theorem \ref{thm:genhom} holds for every 
$\tilde H$ belonging to $\bigcup\limits_{_{\alpha_0,\alpha_1>0,\gamma>2}}\!\!\!\!\!\!\! \Ham$, then it holds for every
$H\in\Hamall$.
\end{prop}

\begin{proof}
Pick a stationary Hamiltonian $H$ in $\Hamall$. Let us fix $R>0$. For any $\theta\in [-R,R]$, we denote by $u_\theta(\cdot,\cdot,\omega)$  the unique Lipschitz solution to \eqref{eq:generalHJ} with initial condition $u_\theta(0,x,\omega)=\theta x$ on $\R$. According to Theorem \ref{appB teo well posed}, there exists a constant $K$, depending on $R>0$, such that $u_\theta(\cdot,\cdot,\omega)$ is $K$-Lipschitz on  $\ccyl$ for every $\omega\in\Omega$ and $|\theta|\leqslant R$. 
Let us now set \ $\tilde H(x,p,\omega):=\max\{H(x,p,\omega),|p|^4-n\}$ \ for all $(x,p,\omega)\in\R\times\R\times\Omega$, with $n\in\N$ chosen large enough so that\ \ $\tilde H\equiv H$\ \ on\ $\R\times [-K,K]\times\Omega$. This impies that the function 
$u_\theta(\cdot,\cdot,\omega)$ is also the unique Lipschitz solution to \eqref{eq:generalHJ} with $\tilde H$ in place of $H$ and initial condition $u_\theta(0,x,\omega)=\theta x$ on $\R$, for every $|\theta|\leqslant R$. Clearly, $\tilde H$ belongs to $\Hamall(\alpha_0,\alpha_1,\gamma)$ for suitable constants $\alpha_0,\alpha_1>0$ and $\gamma\geqslant 4$. Since by hypothesis Theorem \ref{thm:genhom} holds for $\tilde H$, in view of \cite[Lemma 4.1]{DK17} we have, for every $|\theta|\leqslant R$, 
\[
\HV^L(H) (\theta)
=
\liminf_{t\to +\infty}\ \frac{u_\theta(t,0,\omega)}{t}
=
\limsup_{t\to +\infty}\ \frac{u_\theta(t,0,\omega)}{t}
=
\HV^U(H) (\theta)\quad \hbox{almost surely}.
\]
This also implies that 
\begin{equation}\label{eq same HV}
\HV(H)(\theta)=\HV(\tilde H)(\theta)\qquad\hbox{for every $|\theta|\leqslant R$.}
\end{equation} 
By arbitrariness of $R>0$, we derive that equation \eqref{eq:introHJ} homogenizes for $H$ as well. The fact that $\HV(H)$ is superlinear and locally Lipschitz follows from Proposition \ref{appB prop HF}.
\end{proof}

Let us describe our strategy to prove Theorem  \ref{thm:genhom}. In order to prove that $\HV^L(\theta)=\HV^U(\theta)$ for all $\theta\in\R$, 
we will adopt the approach that was taken in \cite{KYZ20, DK22} and substantially  developed in \cite{Y21b}.
It consists in showing the existence of a viscosity Lipschitz solution $u(x,\omega)$ with stationary derivative
for the following stationary equation associated with \eqref{eq:generalHJ}, namely 
\begin{equation}\label{eq:cellPDE}
	a(x,\omega)u''+H(x,u',\omega)=\lambda\qquad\hbox{in $\R$}
		\tag{HJ$_\lambda$}
\end{equation}
for every $\lambda\in\R$ and for $\P$-a.e. $\omega\in\Omega$. With a slight abuse of terminology, we will call such solutions {\em (random) correctors}  in the sequel for the role they play in homogenization.\footnote{The word {\em corrector} is usually used in literature to refer to the function $u(x,\omega)-\theta\,x$\ with\ $\theta:=\EE[u'(0,\omega)]$, see for instance \cite{CaSo17} for a more detailed discussion on the topic.} 
In fact, the following holds.

\begin{prop}\label{prop consequence existence corrector}
Let $\lambda\in\R$ such that equation \eqref{eq:cellPDE} admits a viscosity solution $u$ with stationary gradient. 
Let us set $\theta:=\EE[u'(0,\omega)]$.  Then \ $\HV^L(\theta)=\HV^U(\theta)=\lambda$. 
\end{prop}

\begin{proof}
Let us set $F_\theta(x,\omega):=u(x,\omega)-\theta x$ for all $(x,\omega)\in\R\times\Omega$. Then $F_\theta(\cdot,\omega)$ 
is sublinear for $\P$-a.e. $\omega\in\Omega$, see for instance \cite[Theorem 3.9]{DS09}. Furthermore, it is a viscosity solution of \eqref{eq:cellPDE} with $H(\cdot,\theta+\cdot,\cdot)$ in place of $H$. The assertion follows arguing as in the proof of \cite[Lemma 5.6]{DK22}.
\end{proof}

The first step consists in identifying the set of $\landa\in\R$ for which equation \eqref{eq:cellPDE} can be solved, for every fixed $\omega\in\Omega$. In Section \ref{sec:correctors} we will show that this set is equal to the half-line $[\landa_0(\omega),+\infty)$, where $\landa_0(\omega)$ is a critical constant suitably defined. Furthermore, we will show the existence of a Lipschitz viscosity solution $u_{\lambda}(\cdot,\omega)$ to \eqref{eq:cellPDE} for every $\landa\geqslant\landa_0(\omega)$. Differently from previous works on the subject \cite{KYZ20,DK22,Y21b,DKY23,D23a}, such a critical constant is associated  to the pair $(a,H)$ in an intrinsic way. Indeed, when $a\equiv 0$ or when $H$ is of the form $G(p)+V(x,\omega)$ with $(a,V)$ satisfying the scaled hill condition, $\lambda_0$ is always equal to $\sup_{x\in\R}\min_{p\in\R} H(x,p,\omega)$, while in our case $\lambda_0$ is in general strictly less than this quantity. For the results presented in Section \ref{sec:correctors}, we use in a crucial way the fact that $H$ belongs to $\Ham$ with $\gamma>2$. This will allow us to use known H\"older regularity results for continuous supersolutions of equation \eqref{eq:cellPDE}, see Proposition \ref{prop Holder estimate} for more details. The stationarity of $a$ and $H$ and the ergodicity assumption imply that $\landa_0(\omega)$ is almost surely equal to a constant  $\lambda_0$.  

The subsequent step consists in showing existence of random correctors associated with equation \eqref{eq:cellPDE} for every 
$\lambda\geqslant \lambda_0$. To this aim, we first remark that, due to the almost sure non degenerating condition $a(\cdot,\omega)>0$ in $\R$,  the derivatives of such correctors are in one-to-one correspondence with stationary solutions of the following ODE  
\begin{equation}\label{eq ODE lambda}
a(x,\omega)f'+H(x,f,\omega)=\lambda\qquad\hbox{in $\R$.}
\tag{ODE$_\lambda$}
\end{equation}
We prove that the set $\Sol_\lambda$ of stationary functions $f:\R\times\Omega\to\R$ which solve \eqref{eq ODE lambda} 
almost surely is nonempty  for every $\lambda\geqslant \lambda_0$. This is obtained by showing that the random functions defined as the pointwise maximum and minimum of deterministic solutions of \eqref{eq ODE lambda}, with $\omega$ treated as a fixed parameter, are elements of $\Sol_\lambda$, see Lemma \ref{lemma measurable minimal sol}. 

We then proceed by defining the set-valued map $\Theta:[\lambda_0,+\infty)\to\R$ as
\[
\Theta(\lambda):=\{\EE[f(0,\cdot)]\,:\,f\in\Sol_\lambda\}
\qquad
\hbox{for all $\lambda\geqslant \lambda_0$}.
\] 
The multifunction $\Theta$ is injective, locally equi--compact and upper semicontinuous, in the sense of set-valued analysis.
Furthermore, its image $\D{Im}\left(\Theta\right)$  is contained in the set $D:=\{\theta\in\R\,:\,\HV^L(H)(\theta)=\HV^U(H)(\theta)\}$, in view of Proposition \ref{prop consequence existence corrector}. If $\Theta$ is surjective, then the effective Hamiltonian $\HV(H)$ can be defined as the inverse of $\Theta$, i.e. $\HV(H)(\theta)=\lambda$, where $\lambda$ is the unique real number in $[\lambda_0,+\infty)$ such that $\theta\in\Theta(\lambda)$, for every fixed $\theta\in\R$. Yet, $\Theta$ need not be surjective. To take care of this possibility, we first show that $\D{Im}\left(\Theta\right)$ is a closed subset of $\R$ and, due to the property (v) in Proposition \ref{prop Theta},  any connected component of $\R\setminus\D{Im}\left(\Theta\right)$ is a bounded open interval of the form $(\theta_1,\theta_2)$. Next, we show that this can happen only if $\theta_1$ and $\theta_2$ belong to $\Theta(\lambda)$ for the same $\lambda\geqslant \landa_0$. If $f_1<f_2$ are the stationary random functions in $\Sol_\landa$ such that $\EE[f_i(0,\cdot)]=\theta_i$ for $i\in\{1,2\}$,  the condition that $(\theta_1,\theta_2)$ is in the complement of $\D{Im}\left(\Theta\right)$ implies that the set $\Solfrak(\mu; f_1,f_2)$ of stationary functions in $\Sol_\mu$ almost surely trapped between $f_1$ and $f_2$ is empty for any $\mu\geqslant \lambda_0$, in view of Proposition \ref{prop consequence existence corrector}. This implies that 
\begin{equation}\label{eq flat part}
\lambda\leqslant\HV^L(H)(\theta)\leqslant\HV^U(H)(\theta)\leqslant\lambda\qquad\hbox{for every $\theta\in (\theta_1,\theta_2)$},
\end{equation}
namely, $(\theta_1,\theta_2)$ corresponds to a flat part of the effective Hamiltonian $\HV(H)$. We conclude that $D=\R$ and the effective Hamiltonian $\HV(H)$ can be again defined as the inverse of $\Theta$, as it was explained before. 

In order to prove the first (respectively, last) inequality in \eqref{eq flat part}, we aim to show that, for every fixed $\omega$ in a set of probability $1$ and for every $\eps>0$,  there exists a $C^1$ function $g_\eps$ that bridges $f_2(\cdot,\omega)$ with $f_1(\cdot,\omega)$ (resp.,  $f_1(\cdot,\omega)$ with $f_2(\cdot,\omega)$) in such a way to be a classical supersolution (resp., subsolution) of \eqref{eq ODE lambda} with $\lambda-\eps$ (resp., $\lambda+\eps$) in place of $\lambda$, see Proposition \ref{prop corrector argument} for more details. This argument already appeared in \cite{DK22} and was implemented in  \cite{Y21b, DKY23, D23a} for Hamiltonians of the form $H(x,p,\omega):=G(p)+V(x,\omega)$ by showing that $\inf_\R \left(f_2-f_1\right)=0$ almost surely. The proof of this latter property crucially relies on the assumption that the pair $(a,V)$ satisfies either a scaled hill or valley condition, see Remark \ref{oss lower bound} for further comments. This point is handled here with a novel idea, corresponding to Theorem \ref{teo lower bound}. It exploits the fact that 
$\Solfrak(\mu; f_1,f_2)=\emptyset$ for every $\mu<\lambda$ (resp., $\mu>\lambda$)  to construct a lift which allows to gently descend (resp., ascend)   
from $f_2(\cdot,\omega)$ to $f_1(\cdot,\omega)$ (resp.,  from $f_1(\cdot,\omega)$ to $f_2(\cdot,\omega)$).\smallskip 

We end this section with a disclaimer. As already precised in the previous section, sub and supersolutions of equation \eqref{eq:cellPDE} are understood in the viscosity sense and will be implicitly assumed continuous, without any further specification. We want to remark that, due to the positive sign of the diffusion term $a$, a viscosity supersolution (repectively, subsolution) $u$ to \eqref{eq:cellPDE} that is twice differentiable at $x_0$ will satisfy the  inequality
\[
a(x_0,\omega)u''(x_0,\omega)+H(x_0,u'(x_0),\omega)\leqslant \lambda.\qquad\hbox{(resp.,\  \ $\geqslant \lambda$.)}
\vspace{0.5ex}
\]

\section{Existence of correctors}\label{sec:correctors}

In this section we will characterize the set of admissible $\landa\in\R$ for which the corrector equation 
\begin{equation}\label{eq cellPDE}
	a(x,\omega)u''+H(x,u',\omega)=\lambda\qquad\hbox{in $\R$}
		\tag{HJ$_\lambda$}
\end{equation}
admits solutions. 
In this section we assume that $H$ is a stationary Hamiltonian belonging to the class  $\Ham$ for some fixed constants $\alpha_0,\alpha_1>0$ and $\gamma>2$, for every $\omega\in\Omega$. \smallskip 

We will start by regarding at \eqref{eq cellPDE} as a deterministic equation, where $\omega$ is treated as a fixed parameter. The first step consists in characterizing the set of real $\lambda$ for which equation \eqref{eq cellPDE} admits (deterministic) viscosity solutions. To this aim, for every fixed $\omega\in\Omega$, we define a critical value associated with \eqref{eq cellPDE} defined as follows: 
\begin{equation}\label{def lambda zero}
\lambda_0(\omega):=\inf \left \{ \lambda\in\R\,:\,\hbox{equation \eqref{eq cellPDE} admits a continuous viscosity supersolution}\right\}.
\end{equation}
It can be seen as the lowest possible level under which supersolutions of equation \eqref{eq cellPDE} do not exist at all. 
Such a kind of definition is natural and appears in literature with different names, according to the specific problem for which it is introduced: ergodic constant, Ma\~ne critical value, generalized principal eigenvalue. In the inviscid stationary ergodic setting, it appears in this exact form in \cite{DS11}, 
see also \cite{IsSic20, CDD23} for an analogous definition in the noncompact deterministic setting.\smallskip 

%
%

It is easily seen that $\lambda_0(\omega)$ is uniformly bounded from above on $\Omega$. 
Indeed, the function $u\equiv 0$ on $\R$ is a classical supersolution of \eqref{eq cellPDE} with 
$\lambda:=\sup_{(x,\omega)} H(x,0,\omega)\leqslant \alpha_1$. Hence we derive the upper bound
\begin{equation}\label{eq upper bound lambda zero}
\lambda_0(\omega)\leqslant \alpha_1\qquad\hbox{for all $\omega\in\Omega$.}
\end{equation}

We proceed to show that $\lambda_0(\omega)$ is also uniformly bounded from below and that equation \eqref{eq cellPDE} admits a solution for every $\omega\in\Omega$ and $\lambda\geqslant\lambda_0(\omega)$. For this, we need a preliminary lemma first.  

\begin{lemma}\label{lemma Busemann function}
Let $\omega\in\Omega$ and $\lambda>\lambda_0(\omega)$ be fixed. For every fixed $y\in\R$, there exists a viscosity supersolution $w_y$ of \eqref{eq cellPDE} satisfying $w_y(y)=0$ and 
\begin{equation}\label{claim visco elsewhere}
a(x,\omega)u''+H(x,u',\omega)=\lambda\qquad\hbox{in $\R\setminus\{y\}$}
\end{equation}
in the viscosity sense. In particular 
\begin{equation}\label{eq2 Holder bounds}
|w_y(x)-w_y(z)| \leqslant K |x-z|^{\frac{\gamma-2}{\gamma-1}}\qquad\hbox{for all $x,z\in\R$,}
\end{equation}
where $K=K(\alpha_0,\gamma,\lambda)>0$ is given explicitly by \eqref{eq Holder bound}. 
\end{lemma}

\begin{proof}
Let us denote by $\superSol(\lambda)(\omega)$ the family of continuous viscosity supersolutions of equation 
\eqref{eq cellPDE}. This set is nonempty since $\lambda>\lambda_0(\omega)$. In view of Proposition \ref{prop Holder estimate}, we know that these functions satisfy \eqref{eq2 Holder bounds} for a common constant $K=K(\alpha_0,\gamma,\lambda)>0$ given explicitly by \eqref{eq Holder bound}. 
Let us set 
\[
w_y(x):=\inf\left\{ w\in\CC(\R)\,:\,w\in \superSol(\lambda)(\omega),\ w(y)=0\right\}
\qquad
\hbox{for all $x\in\R$.}
\]
The function $w_y$ is well defined, it satisfies \eqref{eq2 Holder bounds}  and $w_y(y)=0$.  As an infimum of viscosity supersolutions of \eqref{eq cellPDE}, it is itself a supersolution. Let us show that $w_y$ solves \eqref{claim visco elsewhere} in the viscosity sense. If this where not the case, there would exist a $C^2$ strict supertangent $\varphi$ to $w_y$ at a point $x_0\not=y$\,\footnote{i.e., $\varphi>w_y$ in $\R\setminus\{x_0\}$ and $\varphi(x_0)=w_y(x_0)$.} such that \ $a(x_0,\omega)\varphi''(x_0)+H(x_0,\varphi'(x_0),\omega)<\lambda$. By continuity, we can pick $r>0$ small enough so that $|y-x_0|>r$ and 
\begin{equation}\label{eq failed test}
a(x,\omega)\varphi''(x)+H(x,\varphi'(x),\omega)<\lambda
\qquad
\hbox{for all $x\in (x_0-r,x_0+r)$.}
\end{equation}
Choose $\delta>0$ small enough so that 
\begin{equation}\label{eq boundary data}
\varphi(x_0+r)-\delta>w_y(x_0+r)\qquad\hbox{and}\qquad \varphi(x_0-r)-\delta>w_y(x_0-r).
\end{equation}
Let us set \ $\tilde w(x):=\min\{\varphi(x)-\delta,w_y(x)\}$ \ for all $x\in\R$. 
The function $\tilde w$ is the minimum of two viscosity supersolution of \eqref{eq cellPDE} in $(x_0-r,x_0+r)$, in view of \eqref{eq failed test}, and it agrees with $w_y$ in $\R\setminus [x_0-\rho,x_0+\rho]$ for a suitable $0<\rho<r$, in view of \eqref{eq boundary data}. We infer that $\tilde w\in\superSol$ with $\tilde w(y)=0$. But this contradicts the minimality of $w_y$ since $\tilde w(y)=\varphi(y)-\delta<w_y(y)$. 
\end{proof}

With the aid of Lemma \ref{lemma Busemann function}, we can now prove the following existence result. 

\begin{theorem}\label{teo existence solutions}
Let $\omega\in\Omega$ and $\lambda\geqslant \lambda_0(\omega)$ be fixed. Then there exists a viscosity solution $u$ to 
\eqref{eq cellPDE}. Furthermore $u$ is $K$-Lipschitz and of class $C^2$ on $\R$, where the constant  $K=K(\alpha_0,\alpha_1,\gamma,\kappa,\lambda)>0$ is given explicitly by \eqref{eq Lipschitz bound}. In particular, 
$\lambda_0(\omega)\geqslant \min H\geqslant -{1}/{\alpha_0}$. 

\begin{proof}
Let us first assume $\lambda>\lambda_0(\omega)$. 
Let us pick a sequence of points $(y_n)_n$ in $\R$ with $\lim_n |y_n|=+\infty$ and, for each $n\in\N$, set $w_n(\cdot):=w_{y_n}(\cdot)-w_{y_n}(0)$ on $\R$, where $w_{y_n}$ is the function provided by Lemma \ref{lemma Busemann function}. Accordingly, we know that the functions $w_n$ are equi-H\"older continuous and hence locally equi--bounded on $\R$ since $w_n(0)=0$ for all $n\in\N$. By the Ascoli-Arzel\`a Theorem, up to extracting a subsequence, the functions $w_n$ converge in $\CC(\R)$ to a limit function $u$. For every fixed bounded interval $I$, the functions $w_n$ are viscosity solutions of \eqref{eq cellPDE} for all $n\in\N$ big enough since $|y_n|\to +\infty$ as $n\to+\infty$, and so is $u$ by the stability of the notion of viscosity solution. We can now apply Proposition \ref{prop regularity solutions} to infer that $u$ is $K$--Lipschitz continuous and of class $C^2$ on $\R$. In particular
\[
\lambda
= 
a(x,\omega)u''+H(x,u',\omega) 
\geqslant 
a(x,\omega)u''+\min H
\qquad
\hbox{for all $x\in\R$.}
\]
Being $u'$ of class $C^1$ and bounded, we must have $\inf_\R |u''(x)|=0$, hence $\lambda\geqslant \min H$. This readily implies $\lambda_0(\omega)\geqslant \min H\geqslant -1/\alpha_0$.  

Let us now choose a decreasing sequence of real numbers $(\lambda_n)_n$ converging to $\lambda_0(\omega)$ and, for each $n\in\N$, let $u_n$ be a viscosity solution of \eqref{eq cellPDE} with $u_n(0)=0$. In view of Proposition \ref{prop regularity solutions}, the functions $(u_n)_n$ are equi-Lipschitz and thus locally equi--bounded on $\R$, hence they converge, along a subsequence, to a function $u$ in $\CC(\R)$. By stability, $u$ is a solution of \eqref{eq cellPDE} with $\lambda:=\lambda_0(\omega)$. The Lipschitz and regularity properties of $u$ are again a consequence of Proposition \ref{prop regularity solutions}.
\end{proof}
\end{theorem}

The main output of the next result is that the function $\lambda_0(\cdot)$ is almost surely equal to a constant, that will be denoted by $\lambda_0$ in the sequel.

\begin{lemma}\label{lemma lambda zero}
The random variable $\lambda_0:\Omega\to \R$ is measurable and stationary. In particular, it is almost surely constant. 
\end{lemma}

\begin{proof}
If $u\in\CC(\R)$ is a viscosity supersolution of \eqref{eq cellPDE} for some $\omega\in\Omega$ and $\lambda\geqslant \lambda_0(\omega)$, then the functions $u(\cdot+z)$ is a viscosity supersolution of \eqref{eq cellPDE} with $\tau_z\omega$ in place of $\omega$ by the stationarity of $H$ and $a$. By its very definition, we infer that $\lambda_0$ is a stationary function. 

Let us show that  $\lambda_0:\Omega\to \R$ is measurable. Since the probability measure $\P$ is complete on $(\Omega,\F)$, it is enough to show that, for every fixed $\eps>0$, there exists a set $F\in \F$ with $\P(\Omega\setminus F)<\eps$ such that the restriction of 
$\lambda_0$ to $F$ is measurable.  To this aim, we notice that the measure $\P$ is inner regular on $(\Omega,\F)$, see \cite[Theorem 1.3]{Bill99}, hence it is a Radon measure. By applying Lusin's Theorem \cite{LusinThm} to the random variables $a:\Omega\to\CC(\R)$ and $H:\Omega\to\CC(\R\times\R)$, we infer that there exists a closed set $F\subseteq \Omega$ with $\P(\Omega\setminus F)<\eps$ such that $a_{| F}:F\to\CC(\R)$ and 
$H_{| F}:F\to\CC(\R\times\R)$ are continuous. 
We claim that $F\ni\omega\mapsto \lambda_0(\omega)\in \R$ is lower semicontinuous. 
Indeed, let $(\omega_n)_{n\in\N}$ be a sequence converging to some $\omega_0$ in $F$. For each $n\in\N$, let $u_n$ be a solution of \eqref{eq cellPDE} with $\omega_n$ and $\lambda(\omega_n)$ in place of $\omega$ and $\lambda$. Let us furthermore assume that $u_n(0)=0$ for all $n\in\N$. From \eqref{eq upper bound lambda zero} and Proposition \ref{prop regularity solutions} we derive that the functions $u_n$ are equi-Lipschitz and locally equi-bounded in $\R$. Let us extract a subsequence such that $\liminf_n \lambda (\omega_n)=\lim_k \lambda(\omega_{n_k})=:\tilde\lambda$ and 
$\big(u_{n_k}\big)_k$ converges to a function $u$ in $C(\R)$. Since $a(\cdot,\omega_n)\to a(\cdot,\omega_0)$ in 
$\CC(\R)$ and $H(\cdot,\cdot,\omega_n)\to H(\cdot,\cdot,\omega_0)$ in $\CC(\R\times\R)$, we derive by stability that $u$ solves \eqref{eq:cellPDE} with $\omega:=\omega_0$ and $\lambda:=\tilde\lambda$ in the viscosity sense. By definition of $\lambda_0(\omega_0)$, we conclude that $\tilde\lambda\geqslant \lambda(\omega_0)$, i.e., $\liminf_n \lambda(\omega_n)\geqslant \lambda_0(\omega_0)$ as it was claimed. 
\end{proof}

We now take advantage of the fact that the diffusion coefficient $a$ is strictly positive to remark that the derivatives of viscosity solutions to \eqref{eq cellPDE} are in one-to-one correspondence with classical solutions of the  ODE
%
%
\begin{equation}\label{eq2 ODE lambda}
a(x,\omega)f'+H(x,f,\omega)=\lambda\qquad\hbox{in $\R$}
\tag{ODE$_\lambda$}
\end{equation}
The precise statement is the following.
\begin{prop}\label{prop PDE to ODE}
Let $\omega\in\Omega$ and $\lambda\geqslant\lambda_0(\omega)$ be fixed. Then $f$ is a $C^1$ solution of  \eqref{eq2 ODE lambda} if and only if there exists a viscosity solution $u$ of \eqref{eq cellPDE} with $u'=f$. In particular, there exists a constant  $K=K(\alpha_0,\alpha_1,\gamma,\kappa,\lambda)>0$, only depending on 
$\alpha_0,\alpha_1,\gamma,\kappa,\lambda$ through \eqref{eq Lipschitz bound}, such that any $C^1$ solution $f$ of \eqref{eq2 ODE lambda} satisfies \ $\|f\|_\infty\leqslant K$. 
\end{prop}

\begin{proof}
If $f$ is a $C^1$ solution of \eqref{eq2 ODE lambda}, then $u(x):=\int_0^x f(z)\, dz$ is a (classical) solution of \eqref{eq cellPDE}. Conversely, if $u$ is a viscosity solution of \eqref{eq cellPDE}, then $u$ is of class $C^2$ by Proposition \ref{prop regularity solutions} and $f:=u'$ is a $C^1$ solution of \eqref{eq2 ODE lambda}. The remainder of the statement is a direct consequence of Proposition \ref{prop regularity solutions}.
\end{proof}

For every $\lambda\geqslant \lambda_0$ and $\omega\in\Omega$, let us set
\[
\Sol_{\lambda}(\omega):=\left\{f\in \CC^1(\R)\,:\,f\ \hbox{solves \eqref{eq2 ODE lambda}}\,\right\},
\]
where we agree that the above set is empty when $\lambda<\lambda_0(\omega)$. 
\begin{prop}\label{prop minimal solutions}
For every $\lambda\geqslant \lambda_0$ and almost every $\omega\in\Omega$, the set $\Sol_{\lambda}(\omega)$ is nonempty and compact in $\CC(\R)$.
\end{prop}

\begin{proof}
The fact that $\Sol_{\lambda }(\omega)$ is almost surely nonempty is a direct consequence of Theorem \ref{teo existence solutions}, Lemma  
\ref{lemma lambda zero} and Proposition \ref{prop PDE to ODE}. Let us prove compactness. Take a sequence $(f_n)_{n\in\N}$ in 
$\Sol_{\lambda }(\omega)$. This sequence is equi-bounded in view of Proposition \ref{prop PDE to ODE}. The fact that each $f_n$ solves \eqref{eq2 ODE lambda} and $a(\,\cdot\,,\omega) > 0$ on $\R$ implies that the sequence $(f_n)_n$ is locally equi-Lipschitz in $\R$, hence it converges, up to subsequences, to a function $f$ in $\CC(\R)$ by the Arzel\'a-Ascoli Theorem. 
Again by \eqref{eq2 ODE lambda}, the derivatives $(f'_n)_{n\in\N}$ form a Cauchy sequence in $\CC(\R)$,  hence the functions $f_n$ actually converge to $f$ in the local $\CC^1$ topology and $f$ solves \eqref{eq2 ODE lambda}.
\end{proof}

For every $\omega\in\Omega$ and $\lambda\geqslant \lambda_0$, let us set 
\begin{equation}\label{eq minimal solutions}
\ul f_{\lambda }(x,\omega):=\inf_{f\in\Sol_{\lambda }(\omega)} f(x)
\quad\text{and}\quad
\ol f_{\lambda }(x,\omega):=\sup_{f\in\Sol_{\lambda }(\omega)} f(x)
\qquad
\hbox{for all $x\in\R$},
\end{equation}
where we agree to set $\ul f_{\lambda }(\cdot,\omega)\equiv \ol f_{\lambda }(\cdot,\omega)\equiv 0$ when $\Sol_{\lambda }(\omega)=\emptyset$. 
The next lemma shows in particular that $\ul f_{\lambda}\leqslant \ol f_{\lambda}$ almost surely, with equality possibly holding only when $\lambda=\lambda_0$. In fact, the following monotonicity property holds. 

\begin{lemma}\label{lemma monotonicity f lambda}
Let $\omega\in\Omega$ be such that $\lambda_0(\omega)=\lambda_0$. Then 
\begin{equation*}
\ul f_{\mu}(\,\cdot\,,\omega)
<
\ul f_{\lambda}(\,\cdot\,,\omega)
\leqslant
\ol f_{\lambda }(\,\cdot\,,\omega)
<
\ol f_{\mu }(\,\cdot\,,\omega)\quad\hbox{in $\R$}
\qquad
\hbox{for all $\mu>\lambda\geqslant \lambda_0$.} 
\end{equation*}
\end{lemma}

\begin{proof}
Let us fix $\mu>\lambda\geqslant \lambda_0$. Choose 
$R>  \|\ol f_{\lambda}\|_\infty$, where the latter quantity is finite due to Proposition \ref{prop sharp bounds}.  
By the fact that $H\in\Ham$, we can find $p_1>R$ such that \ 
$\inf_{x\in\R} H(x,p_1,\omega)>\mu$. 
We can apply Lemma \ref{lem:inbetw} with $M(\cdot)=p_1$, $m(\cdot):=\ol{f}_\lambda(\cdot,\omega)$ and $G:=H-\mu$, to deduce the existence of a solution  $f\in\CC^1(\R)$ 
of   \eqref{eq2 ODE lambda} with $\mu$ in place of $\lambda$ satisfying $\ol{f}_\lambda(\cdot,\omega)<f<p_1$ on $\R$.  
This implies that $\ol f^+_\lambda(\cdot,\omega)<\ol f^+_\mu(\cdot,\omega)$ in $\R$ by definition of $\ol f^+_\mu$. The other inequality can be proved in a similar way. 
\end{proof}

The next lemma yields that the functions $\ul f_{\lambda }, \ol f_{\lambda }$ are stationary solutions to \eqref{eq2 ODE lambda}. 
To prove joint measurability, we use again the fact that $\Omega$ is a Polish space, $\F$ is its Borel $\sigma$-algebra  and $\P$ is a complete probability measure.\smallskip 

\begin{lemma}\label{lemma measurable minimal sol}
Let $\lambda\geqslant \lambda_0$. Then the functions $\ul f_{\lambda }, \ol f_{\lambda }: \R\times\Omega\to\R$ are jointly measurable and stationary. Moreover, $\ul f_{\lambda }(\,\cdot\,,\omega),\, \ol f_{\lambda }(\,\cdot\,,\omega)\in \Sol_\lambda(\omega)$ for every $\omega\in\Omega$.
\end{lemma}

\begin{proof}
The last assertion is a direct consequence of Proposition \ref{prop minimal solutions} and Lemma \ref{appendix A lemma lattice}. The stationarity property of $\ul f$, $\ol f$ is a consequence of the fact that, for every fixed $z\in\R$ and $\omega\in\Omega$, $f\in \Sol_\lambda(\omega)$ if and only if $f(\,\cdot\,+z)\in \Sol_\lambda(\tau_z\omega)$ since $a(\,\cdot\,+z,\omega)=a(\,\cdot\,,\tau_z\omega)$ and $H(\,\cdot\, + z,\cdot,\omega)=H(\,\cdot\,,\cdot,\tau_z\omega)$. 

Let us prove that $\ul f:\R\times\Omega\to\R$ is measurable with respect to the product $\sigma$-algebra ${\mathcal B}\otimes\F$. This is equivalent to showing that $\Omega\ni \omega\mapsto \ul f(\,\cdot\,,\omega)\in\CC(\R)$ is a random variable from $(\Omega,\F)$ to the Polish space $\CC(\R)$ endowed with its Borel $\sigma$-algebra, see for instance  \cite[Proposition 2.1]{DS09}. 
Since the probability measure $\P$ is complete on $(\Omega,\F)$, it is enough to show that, for every fixed $\eps>0$, there exists a set $F\in \F$ with $\P(\Omega\setminus F)<\eps$ such that the restriction 
$\ul f$ to $F$ is a random variable from $F$ to $\CC(\R)$.  To this aim, we notice that the measure $\P$ is inner regular on $(\Omega,\F)$, see \cite[Theorem 1.3]{Bill99}, hence it is a Radon measure. By applying Lusin's Theorem \cite{LusinThm} to the random variables $a:\Omega\to\CC(\R)$ and $H:\Omega\to\CC(\R\times\R)$, we infer that there exists a compact set $F\subseteq \Omega$ with $\P(\Omega\setminus F)<\eps$ such that $a_{| F},H_{| F}:F\to\CC(\R\times\R)$ are continuous. We claim that $\ul f:\R\times F\to\R$ is lower semicontinuous. 
Indeed, let $(x_n,\omega_n)_{n\in\N}$ be a sequence converging to some $(x_0,\omega_0)$ in $\R\times F$. 
By the continuity of $a$ on $\R\times F$, we have that $\min_{J\times F} a>0$ for every compact interval $J\subset \R$. This implies that 
the functions  $\ul f(\,\cdot\,,\omega_n)$ are locally equi-Lipschitz on $\R$. Since they are also equi-bounded on $\R$ in view of Proposition \ref{prop PDE to ODE}, by the Arzel\`a-Ascoli Theorem, we can extract a subsequence $\big(x_{n_k},\omega_{n_k}\big)_{k\in\N}$ such that $\liminf_n \ul f(x_n,\omega_n)=\lim_k \ul f(x_{n_k},\omega_{n_k})$ and 
$\ul f(\,\cdot\,,\omega_{n_k})$ converge to some $f$ in $\CC(\R)$. Since each function $\ul f(\,\cdot\,,\omega_{n})$ is a solution to \eqref{eq2 ODE lambda} with $\omega:=\omega_n$ and $a(\,\cdot\,,\omega_n)\to a(\,\cdot\,,\omega_0)$, $H(\,\cdot,\cdot,\omega_n)\to H(\,\cdot,\cdot,\omega_0)$ in $\CC(\R\times\R)$, an argument analogous to the one used in the proof of Proposition  \ref{prop minimal solutions} shows that the functions $\ul f(\,\cdot\,,\omega_{n_k})$ actually converge to $f$ in the local $\CC^1$ topology. This readily implies that $f\in \Sol_\lambda(\omega_0)$. By the definition of $\ul f$, we conclude that $\ul f(\,\cdot\,,\omega)\leqslant f$, in particular
\[
\ul f(x_0,\omega_0)
\leqslant
f(x_0)
=
\lim_k \ul f (x_{n_k},\omega_{n_k})
=
\liminf_n \ul f (x_n,\omega_{n}),
\]
proving the asserted lower semicontinuity property of $\underline f$. This implies that $\ul f_{|F}:F\to\CC(\R)$ is measurable (see, e.g., \cite[Proposition 2.1]{DS09}).
Via a similar argument, one can show that $\ol f:\R\times F\to\R$ is upper semicontinuous. 
\end{proof}

We end this section by stating a result which give a more precise description of the bounds enjoyed by stationary solutions of   \eqref{eq2 ODE lambda}. It is a direct consequence of Corollary \ref{cor:lelam}. 
%

\begin{prop}\label{prop sharp bounds}
Let $\lambda\geqslant \lambda_0$ and let $f:\R\times\Omega\to\R$ be a stationary function such that $f(\cdot,\omega)$  is a $C^1$ solution of   \eqref{eq2 ODE lambda}, for almost every $\omega$. Then, for every $\omega$ in a set of probability 1, 
we have 
\begin{equation*}
\sup_{\mu>\lambda}p^-_\mu
\leqslant
f(x,\omega)
\leqslant
\inf_{\mu>\lambda}p^+_\mu
\qquad\hbox{for all $x\in\R$.}
\end{equation*}
where $p_\mu^\pm$ are constants depending on $\mu\in\R$ defined  almost surely as follows:
\begin{align*}
	p_\lambda^-:=\inf_{x\in\R}\inf\{p\in\R:\ H(p,x,\omega)\leqslant \lambda\} \quad\text{and}\quad p^+_\lambda:=\sup_{x\in\R}\sup\{p\in\R:\ H(p,x,\omega)\leqslant \lambda\}.
\end{align*}
\end{prop}


\section{Bridging stationary solutions}\label{sec:bridging}

This section is devoted to the proof of Theorem \ref{teo lower bound}, where we provide a novel  argument to suitably bridge a pair of distinct stationary functions $f_1<f_2$ which solve the following ODE for the same value of $\lambda$: 
\begin{equation}\label{eq ODE lambda}
a(x,\omega)f'+H(x,f,\omega)=\lambda\qquad\hbox{in $\R$}.
\tag{ODE$_\lambda$}
\end{equation}
It exploits the absence of stationary functions trapped between $f_1$ and $f_2$ which solve  \eqref{eq ODE lambda} for values of the parameter at the right-hand side lower (respectively, greater) than $\lambda$ in order to construct a lift which allows to gently descend from $f_2$ to $f_1$ (resp., to ascend from $f_1$ to $f_2$). 
The absence of such stationary solutions will be used in the proof of Proposition \ref{prop D}, which can be regarded as a crucial step in order to establish our homogenization result stated in Theorem \ref{thm:genhom}. 
\smallskip

Most of the results of this section holds under a very general set of assumptions on $H$, such as conditions (G1)-(G2) in Appendix \ref{app:ODE}. However, for the sake of readability and consistency with the rest of the paper, we shall assume that $H\in\Ham$ for some constants $\alpha_0,\alpha_1>0$ and $\gamma>2$.\smallskip

Let us denote by $\Sol_\lambda$ the family of essentially bounded and jointly measurable stationary functions 
$f:\R\times\Omega\to\R$ such that $f(\cdot,\omega)$ is a (bounded) $C^1$ solution of \eqref{eq ODE lambda} almost surely. 
We know that this set is nonempty whenever $\lambda\geqslant \lambda_0$, in view of Proposition \ref{prop minimal solutions} 
and Lemmas \ref{lemma monotonicity f lambda} and \ref{lemma measurable minimal sol}.\smallskip

We distill in the next statement an arguments that already appeared in \cite{DK22} and which will be needed 
for the proof of Theorem \ref{teo lower bound}.

\begin{prop}\label{prop corrector argument}
Let $f_1,f_2$ be functions belonging to $\Sol_\lambda$ for some $\lambda\in\R$ with 
with $f_1<f_2$ on $\R$ almost surely. Let us set $\theta_i:=\EE[f(0,\omega)]$ for $i\in\{1,2\}$. 
\begin{itemize}
\item[(i)] Let us assume that there exist $\eps>0$ and a set $\Omega_\eps$ of positive probability such that, for every $\omega\in\Omega_\eps$, there exists a $C^1$ and bounded function $g_\eps$ satisfying  
\[
a(x,\omega)g_\eps'+H(x,g_\eps,\omega)\leqslant\lambda+\eps\qquad\hbox{in $\R$}
\]
and such that $g_\eps(\cdot)=f_1(\cdot,\omega)$ in $(-\infty, -L)$ and $g_\eps(\cdot)=f_2(\cdot,\omega)$ in $(L,+\infty)$ for some $L>0$. Then \ $\HV^U(H)(\theta)\leqslant \lambda+\eps$\quad for every $\theta\in (\theta_1,\theta_2)$.\smallskip
\item[(ii)] Let us assume that there exist $\eps>0$ and a set $\Omega_\eps$  of positive probability such that, for every $\omega\in\Omega_\eps$, there exists a $C^1$ and bounded function $g_\eps$ satisfying  
\[
a(x,\omega)g_\eps'+H(x,g_\eps,\omega)\geqslant \lambda-\eps\qquad\hbox{in $\R$}
\]
and such that $g_\eps(\cdot)=f_2(\cdot,\omega)$ in $(-\infty, -L)$ and $g_\eps(\cdot)=f_1(\cdot,\omega)$ in $(L,+\infty)$ for some $L>0$. Then \ $\HV^L(H)(\theta)\geqslant \lambda-\eps$\quad for every $\theta\in (\theta_1,\theta_2)$.
\end{itemize}
\end{prop}

\begin{proof}
The proof is based on an argument already used in \cite[Lemma 5.6]{DK22}. We shall only prove (i), being the proof of (ii) analogous. Let us fix $\theta\in (\theta_1,\theta_2)$ and choose a set $\hat\Omega$ of probability 1 such that, for every $\omega\in\hat\Omega$, both the limits in \eqref{eq:infsup}  and the following ones (by Birkhoff Ergodic Theorem) hold:
\begin{align}
&\lim_{x\to+\infty}\frac{1}{x}\int_{0}^x{f_2}(s,\omega)\,ds=\EE[f_2(0,\omega)] > \theta\label{eq1 Birkhoff}\\
&\lim_{x\to-\infty}\frac{1}{|x|}\int_{x}^{0}{f_1}(s,\omega)\,ds=\EE[f_1(0,\omega)] < \theta.\label{eq2 Birkhoff}
\end{align}
Let us fix  $\omega\in\hat\Omega\cap\Omega_\eps$ and set $w_\eps(x):=\int_0^x g_\eps(z)\, dz$ for all $x\in\R$, where $g_\eps$ is as in statement (i). Let us set
\[
\tilde w_\eps(t,x)=(\lambda+\eps)t+w_\eps(x)+M_\eps\qquad\hbox{for all $(t,x)\in \ccyl$},
\]
where the constant $M_\eps$ will be chosen later to ensure that $\tilde w_\eps(0,x)\geqslant \theta x$ for all $x\in\R$. The function $\tilde w_\eps$ is a supersolution of \eqref{eq:generalHJ}. Indeed, 
	\begin{align*}
		\partial_t\tilde w_\eps
=
		(\lambda+\eps)\ge
a(x,\omega)g'_\eps + H(x, g_\eps,\omega)
=	
a(x,\omega)\partial_{xx}^2\tilde w_\eps + H(x,\partial_x \tilde w_\eps,\omega).
\end{align*}
In view of \eqref{eq1 Birkhoff} and \eqref{eq2 Birkhoff},  we can pick $M_\eps$ large enough so that $\tilde w_\eps(0,x)\geqslant \theta x$ for all $x\in\R$.  By the comparison principle, $\tilde w_\eps(t,x)\geqslant u_\theta(t,x,\omega)$ on $\ccyl$ and, hence,
	\[\HV^U(H)(\theta)=\limsup_{t\to+\infty}\frac{u_\theta(t,0,\omega)}{t}\leqslant\limsup_{t\to+\infty}\frac{\tilde w_\eps(t,0)}{t}=\lambda+\eps.\qedhere
\] 
\end{proof}

Let us now prove the main result of this section.

\begin{theorem}\label{teo lower bound}
Let $f_1,f_2$ be functions belonging to $\Sol_\lambda$ for some $\lambda\in\R$. 
Let us assume that $f_1<f_2$ on $\R$ almost surely and set $\theta_i:=\EE[f_i(0,\omega)]$ for $i\in\{1,2\}$. 
For any $\mu\in\R$, let us denote by $\Solfrak(\mu;f_1,f_2)$ the family of bounded and $C^1$ stationary solutions $f$ of \eqref{eq ODE lambda} with $\mu$ in place of $\lambda$ that satisfy $f_1< f <f_2$ in $\R$, almost surely. 
\begin{itemize}
\item[\em (i)] 
If   $\Solfrak(\mu;f_1,f_2)=\emptyset$ for any $\mu<\lambda$, then 
\[
\HV^L(H)(\theta)\geqslant \lambda\qquad\hbox{for all $\theta\in (\theta_1,\theta_2)$.}
\]
\item[\em (ii)] 
If   $\Solfrak(\mu;f_1,f_2)=\emptyset$ for any $\mu>\lambda$, then 
\[
\HV^U(H)(\theta)\leqslant \lambda\qquad\hbox{for all $\theta\in (\theta_1,\theta_2)$.}
\]

\end{itemize}
\end{theorem}

\begin{proof}
(i) Let us fix $\mu<\lambda$ and $\eps>0$.  In view of Proposition \ref{prop corrector argument}, it suffices to show that, for every fixed $\omega$ in a set of probability 1, there exists a $C^1$ and bounded function $g_\eps$ such that 
\begin{equation}\label{claim lower bound}
a(x,\omega)g_\eps'+H(x,g_\eps,\omega)> \mu-2\eps\qquad\hbox{in $\R$}
\end{equation}
and satisfying $g_\eps(\cdot)=f_2(\cdot,\omega)$\  in $(-\infty,- L ]$,\  $g_\eps(\cdot)=f_1(\cdot,\omega)$\ in $[ L ,+\infty)$ for $ L >0$ large enough.\smallskip

To this aim, pick $R>1+\max\{\|f_1\|_{L^\infty(\R\times\Omega)} ,\|f_2\|_{L^\infty(\R\times\Omega)}\}$ and denote by $C_R$ the Lipschitz constant of $H$ on $\R\times[-R,R]\times\R$. Let us pick $\omega\in\Omega_0:=\{\omega\,:\,f_1(\cdot,\omega)<f_2(\cdot,\omega)\ \hbox{in $\R$}\}$. 
For  fixed $n\in\N$, let us denote by $\tilde g_n:[-n,b)\to\R$ the unique solution of the following ODE 
\begin{equation}\label{eq ODE lower bound}
a(x,\omega)f'+H(x,f,\omega)=\mu\qquad\hbox{in $[-n,b)$}
\end{equation}
satisfying $\tilde g_n(-n)=f_2(-n,\omega)$, for some $b>-n$. The existence of such a $b$ is guaranteed by the  classical Cauchy-Lipschitz Theorem. Let us denote by $I$ the maximal interval of the form $(-n,b)$ where such a $\tilde g_n$ is defined. 
It is easily seen that $\tilde g_n<f_2(\cdot,\omega)$ on $I$: since $\mu<\lambda$ and $a(\cdot,\omega)>0$ on $\R$, we have $g_n'(x_0)<f_2'(x_0,\omega)$ at any possible point $x_0>-n$ where $g_n(x_0)=f_2(x_0,\omega)$, so in fact this equality never occurs. Let us denote by $y_n:=\sup\{y\in I\,:\,\tilde g_n>f_1(\cdot,\omega)\ \hbox{in $(-n,y)$}\}$. 
 We claim that there exists $n\in\N$ such that $y_n$ is finite. If this were not the case, then $f_1(\cdot,\omega)<\tilde g_n<f_2(\cdot,\omega)$ in $(-n,+\infty)$. Being $a(\cdot,\omega)>0$, we infer from 
\eqref{eq ODE lower bound} that the functions $\tilde g_n$ are locally equi-Lipchitz and equi-bounded on their domain of definition. Up to extracting a subsequence, we derive that the functions $\tilde g_n$ locally uniformly converge  to a function $\tilde g$ on $\R$, and also locally in $C^1$ norm being each $\tilde g_n$ a solution of \eqref{eq ODE lower bound}. Hence $\tilde g$ is a solution of \eqref{eq ODE lower bound} on $\R$ satisfying 
$f_1(\cdot,\omega)\leqslant \tilde g \leqslant f_2(\cdot,\omega)$ \ in $\R$. The same kind of argument employed above shows that in fact these inequalities are actually strict, namely $\tilde g$ satisfies 
\begin{equation}\label{eq bounds}
f_1(\cdot,\omega)< \tilde g < f_2(\cdot,\omega)\qquad\hbox{in $\R$.}
\end{equation}
Let us denote by $\Solfrak(\mu; f_1,f_2)(\omega)$ the set of deterministic $C^1$ functions satisfying \eqref{eq bounds} which solve 
equation \eqref{eq ODE lower bound} in $\R$. By the stationary character of $a,\, H,\, f_1,\,f_2$, it is easily checked that $\tilde g\in \Solfrak(\mu; f_1,f_2)(\omega)$ if and only if $\tilde g(\cdot+z)\in \Solfrak(\mu; f_1,f_2)(\tau_z\omega)$. By the ergodicity assumption, we derive that the set $\hat\Omega:=\{\omega\in\Omega\,:\,\Solfrak(\mu; f_1,f_2)(\omega)=\emptyset\}$ has either probability 0 or 1. If it has probability 0, for every $\omega\in\Omega$ we set 
\[
\overline g(x,\omega):=\sup\limits_{\tilde g\in\Solfrak(\mu; f_1,f_2)(\omega)} \tilde g(x,\omega)\qquad\hbox{for all $x\in\R$,}
\]
where we agree that $\overline g(\cdot,\omega)\equiv 0$ when $\omega\in\hat\Omega$. 
The function $\overline g:\R\times\Omega\to\R$ is jointly measurable and stationary, and 
$\overline g(\cdot,\omega)\in\Solfrak(\mu; f_1,f_2)(\omega)$ almost surely. This follows by arguing as in the proof of
Lemma \ref{lemma measurable minimal sol}.  Note that the fact that $\overline g$ is stationary depends on the fact that the both the ingredients of  equation \eqref{eq ODE lower bound}, both the functions $f_1,f_2$ are stationary. The fact that $\ol g$ satisfies \eqref{eq bounds} with strict inequalities can be shown by arguing as above, or by invoking Lemma \ref{lem:ordqu}. 
We derive that $\overline g$ belongs to  $\Solfrak(\mu;f_1,f_2)$, contradicting the assumption.  

This implies that $\hat\Omega$ has probability 1, hence for every $\omega\in \hat\Omega$  there exist points 
$\hat x<\hat y$ in $\R$ and a function $g:[\hat x,\hat y]\to\R$ which solves equation \eqref{eq ODE lower bound} in 
$(\hat x,\hat y)$  and satisfies 
\begin{equation*}
f_1(\cdot,\omega)<g<f_2(\cdot,\omega)\quad\hbox{in $(\hat x,\hat y)$},
\quad g(\hat x)=f_2(\hat x,\omega),
\quad g(\hat y)=f_1(\hat y,\omega).
\end{equation*}
Let us extend the function $g$ to the whole $\R$ by setting $g=f_2$ on $(-\infty,\hat x)$ and $g=f_1$ on $(\hat y,+\infty)$. 
Take a sequence of standard even convolution kernels $\rho_n$ supported in $(-1/n,1/n)$ and set $g_n:=\rho_n*g$. Let us pick $r>0$ and choose $n\in\N$ big enough so that $1/n<r$ and $\|g-g_n\|_\infty<1$. We claim that we can choose $n$ big enough such that 
\begin{equation} \label{claim2 lower bound}
a(x,\omega)g_n'+H(x,g_n,\omega)>\mu-\eps\qquad\hbox{in $(\hat x-r,\hat y+r)$.}
\end{equation}
To this aim, first observe that the map $x\mapsto H(x,g(x),\omega)$ is $K$-Lipschitz continuous in 
$(\hat x-2r,\hat y+2r)$, for some constant $K$, due to the fact that $g$ is bounded and locally Lipschitz on $\R$.  For every $x\in (\hat x-r,\hat y+r)$ and $|y|\leqslant 1/n$ we have 
\begin{eqnarray*}
H(x,g_n(x),\omega) 
&\geqslant&
H(x,g(x),\omega)-C_R\|g-g_n\|_{_{L^\infty([\hat x-r,\hat y+r])}}\\
&\geqslant& 
H(x-y,g(x-y),\omega)-C_R\|g-g_n\|_{_{L^\infty([\hat x-r,\hat y+r])}}-\dfrac{K}{n},
\end{eqnarray*}
hence
\begin{equation*}\label{eq inequality 1}
H(x,g_n(x),\omega)
\geqslant
\int_{-1/n}^{1/n} H(x-y,g(x-y),\omega)\,\rho_n(y)\, dy -C_R\|g-g_n\|_{_{L^\infty([\hat x-r,\hat y+r])}}-\dfrac{K}{n}.
\end{equation*}
For all $x\in (\hat x-r,\hat y+r)$ we have
\begin{eqnarray*}
a(x,\omega)g_n'(x)+H(x,g_n(x),\omega)
&\geqslant& 
-C_R\|g-g_n\|_{_{L^\infty([\hat x-r,\hat y+r])}}-\dfrac{K+\kappa \|g'\|_{_{L^\infty([\hat x-2r,\hat y+2r])}}}{n}\\
&+& \int_{-1/n}^{1/n} \Big( a(x-y,\omega)g'(x-y)+H(x-y,g(x-y),\omega)  \Big)\,\rho_n(y)\, dy \\
&>& 
\mu -\eps
\end{eqnarray*}
for $n\in\N$ big enough, where we have used the fact that $g_n\to g$ in $C(\R)$ as $n\to +\infty$. 
Let us now take $\xi\in C^1(\R)$ such that 
\[
0\leqslant \xi \leqslant 1\quad\hbox{in $\R$,}
\quad
\xi\equiv 0\quad\hbox{in $(-\infty,\hat x-r]\cup [\hat y+r,+\infty)$},
\quad
\xi\equiv 1\quad\hbox{in $[\hat x-r/2,\hat y+r/2]$},
\]
and we set $g_\eps(x):=\xi(x) g_n(x)+(1-\xi(x))g(x)$ for all $x\in\R$. We will show that we can choose $n$ in 
$\N\cap (1/r,+\infty)$ big enough so that  $g_\eps$ satisfies \eqref{claim lower bound}. We only need to check it in $(\hat x-r,\hat y+r)$. 
For notational simplicity, we momentarily suppress $(x,\omega)$ from some of the notation below and observe that
\begin{align*}\label{eq:concur}
ag_\eps' + H(x,g_\eps,\omega) &= \xi\big(a g_n' + H(x,g_n,\omega)\big) + (1-\xi)\big(a g ' + H(x,g,\omega )\big)
\nonumber\\
&\quad + \xi\big( H(x,\xi g_n + (1-\xi) g,\omega) -H(x,g_n,\omega)\big) \\
&\quad +
(1-\xi)\big( H(x,\xi g_n + (1-\xi) g,\omega) - H(x, g,\omega )\big) +a\xi'( g_n  -  g) \nonumber\\
&\quad >
\mu-\eps-\big(C_R+\|\xi'\|_\infty\big)\|g-g_n\|_{_{L^\infty([\hat x-r,\hat y+r])}}.
\end{align*}
By choosing $n$ big enough we get the assertion.\smallskip

(ii) We will just sketch the proof, since the argument is analogous to the one presented above.  Let us fix $\mu>\lambda$ and $\eps>0$.  In view of Proposition \ref{prop corrector argument}, it suffices to show that, for every fixed $\omega$ in a set of probability 1, there exists a $C^1$ and bounded function $g_\eps$ such that 
\begin{equation}\label{claim upper bound}
a(x,\omega)g_\eps'+H(x,g_\eps,\omega)< \mu+2\eps\qquad\hbox{in $\R$}
\end{equation}
and satisfying $g_\eps(\cdot)=f_1(\cdot,\omega)$\  in $(-\infty,- L ]$,\  $g_\eps(\cdot)=f_2(\cdot,\omega)$\ in $[ L ,+\infty)$ for $ L >0$ large enough.\smallskip

To this aim, pick $R>1+\max\{\|f_1\|_{L^\infty(\R\times\Omega)} ,\|f_2\|_{L^\infty(\R\times\Omega)}\}$ and denote by $C_R$ the Lipschitz constant of $H$ on $\R\times[-R,R]\times\R$. Let us pick $\omega\in\Omega_0:=\{\omega\,:\,f_1(\cdot,\omega)<f_2(\cdot,\omega)\ \hbox{in $\R$}\}$. 
For  fixed $n\in\N$, let us denote by $\tilde g_n:[-n,b)\to\R$ the unique solution of the following ODE 
\begin{equation}\label{eq ODE upper bound}
a(x,\omega)f'+H(x,f,\omega)=\mu\qquad\hbox{in $[-n,b)$}
\end{equation}
satisfying $\tilde g_n(-n)=f_1(-n,\omega)$, for some $b>-n$. Arguing as in item (i) and by exploiting the assumption 
$\Solfrak(\mu;f_1,f_2)=\emptyset$, we get that there exists a set 
$\hat\Omega$ of probability 1 such that, for every $\omega\in \hat\Omega$,  there exist points 
$\hat x<\hat y$ in $\R$ and a function $g:[\hat x,\hat y]\to\R$ which solves equation \eqref{eq ODE upper bound} in 
$(\hat x,\hat y)$  and satisfies 
\begin{equation*}
f_1(\cdot,\omega)<g<f_2(\cdot,\omega)\quad\hbox{in $(\hat x,\hat y)$},
\quad g(\hat x)=f_1(\hat x,\omega),
\quad g(\hat y)=f_2(\hat y,\omega).
\end{equation*}
We extend the function $g$ to the whole $\R$ by setting $g=f_1$ on $(-\infty,\hat x)$ and $g=f_2$ on $(\hat y,+\infty)$, then we  
take a sequence of standard even convolution kernels $\rho_n$ supported in $(-1/n,1/n)$ and we set $g_n:=\rho_n*g$.
We choose $n$ big enough so that 
\begin{equation*} \label{claim2 upper bound}
a(x,\omega)g_n'+H(x,g_n,\omega)<\mu+\eps\qquad\hbox{in $(\hat x-r,\hat y+r)$.} 
\end{equation*}
Next, we define $g_\eps(x):=\xi(x) g_n(x)+(1-\xi(x))g(x)$ for all $x\in\R$, where $\xi$ is a $C^1$-function on $\R$ such that 
\[
0\leqslant \xi \leqslant 1\quad\hbox{in $\R$,}
\quad
\xi\equiv 0\quad\hbox{in $(-\infty,\hat x-r]\cup [\hat y+r,+\infty)$},
\quad
\xi\equiv 1\quad\hbox{in $[\hat x-r/2,\hat y+r/2]$}.
\]
Arguing as in (i), we infer that we can choose $n$ big enough so that $g_\eps$ satisfies \eqref{claim upper bound}. The proof is complete.  
\end{proof}

\begin{remark}\label{oss lower bound}
In \cite{Y21b, DKY23, D23a} the lower (respectively, upper) bound for $\HV^L(H)$  (resp., $\HV^U(H)$) 
appearing in the statement of Theorem \ref{teo lower bound} was 
proved for Hamiltonians of the form $G(p)+V(x,\omega)$ by showing that $\inf_\R \left(f_2-f_1\right)=0$ almost surely. The proof of this latter property crucially relies on the assumption that the pair $(a,V)$ satisfies a scaled valley (resp., hill) condition, see  \cite[Lemma 4.7]{Y21b}, \cite[Lemma 4.2]{DKY23}, \cite[Lemma 4.3]{D23a}. In the periodic setting,  this condition is met by constant potentials only. It is worth noticing that the hypothesis $\Solfrak(\mu;f_1,f_2)=\emptyset$ for every $\mu\not=\lambda$  herein assumed is always met when $\inf_\R \left(f_2-f_1\right)=0$ almost surely, in view of  Lemma \ref{lem:ordqu}.

\end{remark}

\section{The homogenization result}\label{sec:homogenization}

This section is devoted to the proof of the homogenization result stated in Theorem \ref{thm:genhom}. To get to it, 
we need to prove first a series of preparatory results. Throughout the section, we will assume $G\in\Ham$ for constants $\alpha_0,\alpha_1>0$ and $\gamma>2$. \smallskip 

We start by recalling that we have denoted by $\Sol_\lambda$ the family of essentially bounded and jointly measurable stationary functions $f:\R\times\Omega\to\R$ such that $f(\cdot,\omega)$ is a (bounded) $C^1$ solution of the following ODE:
\begin{equation}\label{eq2 ODE lambda}
a(x,\omega)f'+H(x,f,\omega)=\lambda\qquad\hbox{in $\R$}
\tag{ODE$_\lambda$}
\end{equation}
This set is nonempty whenever $\lambda\geqslant \lambda_0$, according to the results presented in Section \ref{sec:correctors}. \smallskip

We define the set-valued map $\Theta: [\lambda_0,+\infty)\to\parts(\R)$ as follows:
\[
\Theta(\lambda):=\{\EE[f(0,\cdot)]\,:\,f\in\Sol_\lambda\}
\qquad
\hbox{for all $\lambda\geqslant \lambda_0$}.
\] 
The following holds. 

\begin{prop}\label{prop Theta}
The set-valued map $\Theta: [\lambda_0,+\infty)\to\parts(\R)$  satisfies the following properties:
\begin{itemize}
\item[\em (i)] \quad $\Theta(\lambda)$ is a nonempty compact set for every $\lambda\geqslant \lambda_0$;\smallskip
\item[\em (ii)] \quad $\Theta(\cdot)$ is locally equi-compact;\smallskip
\item[\em (iii)] \quad $\Theta(\cdot)$ is upper semicontinuous;\smallskip
\item[\em (iv)] \quad $\Theta(\lambda_1)\cap\Theta(\lambda_2)=\emptyset$ \quad if $\lambda_1\not=\lambda_2$;\smallskip
\item[\em (v)] \ $\displaystyle\lim_{\lambda\to +\infty} \max\Theta(\lambda)=+\infty$, \quad  
$\displaystyle\lim_{\lambda\to -\infty} \min\Theta(\lambda)=-\infty$.
\end{itemize}
\end{prop}

\begin{proof}
The fact that $\Theta(\lambda)$ is nonempty for every $\lambda\geqslant \lambda_0$ has been already remarked above. 
According to Proposition \ref{prop sharp bounds}, for every $\lambda\geqslant \lambda_0$ we have 
\[
\sup_{\mu>\lambda} p_\mu^-\leqslant \theta \leqslant \inf_{\mu>\lambda} p_\mu^+.
\]
This shows that $\Theta(\cdot)$ is locally equi-bounded. The fact that $\Theta(\cdot)$ is locally equi-compact is a consequence of the upper semicontinuity of $\Theta(\cdot)$, that we proceed to show below. 

(iii) Let $\lambda_n\to \lambda$ in $[\lambda_0,+\infty)$ and $\theta_n\in\Theta(\lambda_n)$ for each $n\in\N$ with $\theta_n\to\theta$. We aim to show that $\theta\in\Theta(\lambda)$. Up to extracting a subsequence, if necessary, we can furthermore assume that $(\theta_n)_n$ is monotone, let us say nondecreasing for definitiveness. Then, according to Lemma \ref{lem:order}, we must have 
\[
f_n(\cdot,\omega)\leqslant f_{n+1}(\cdot,\omega)\quad\hbox{in $\R$}\qquad \hbox{a.s. in $\Omega$\qquad for all $n\in\N$.}
\]
According to Corollary \ref{cor:lelam}, the functions $\big(f_n(\cdot,\omega)\big)_n$ are equi-bounded in $\R$ for every $\omega\in\Omega$, and hence they are also locally equi-Lipschitz since they solve \eqref{eq2 ODE lambda} and $a(\cdot,\omega)>0$ in $\R$. We conclude that 
\[
f_n(\cdot,\omega)\to f(\cdot,\omega)=\sup_n f_n(\cdot,\omega) \qquad\hbox{for all $\omega\in\Omega$}
\]
in the local $C^1$ topology on $\R$. Since $\lambda_n\to\lambda$, it is easily seen that $f\in\Sol_\lambda$. By the Dominated Convergence Theorem, we conclude that 
\[
\theta
=
\lim_n \theta_n
=
\lim_n \EE[f_n(0,\cdot)]
=
\EE[f(0,\cdot)],
\]
thus showing that $\theta\in\Theta(\lambda)$. 

(iv) Obvious in view of Lemma \ref{lem:ordqu}. 

(v)  Let us prove the first inequality. 
By the fact that $H\in\Ham$, for every fixed $R>0$ we can find $p_2>p_1>R$ and $\lambda=\lambda(R)\in\R$ such that \ 
$\inf_{(x,\omega)} H(x,p_2,\omega)>\lambda>\sup_{(x,\omega)} H(x,p_1,\omega)>\lambda_0$. 
For every fixed $\omega\in\Omega$, we can apply Lemma \ref{lem:inbetw} with $M(\cdot)=p_2$, $m(\cdot):=p_1$ and $G:=H-\lambda$, to deduce the existence of a solution  $f\in\CC^1(\R)$ 
of   \eqref{eq ODE lambda} satisfying $p_1<f(\cdot)<p_2$ on $\R$.  
This implies that $\ol f_\lambda(\cdot,\omega)>p_1$ in $\R$ almost surely, yielding $\max\Theta(\lambda)>p_1>R$. Since the function $\max \Theta(\cdot)$ is also strictly increasing on $[\lambda_0,+\infty)$ in view of Lemma \ref{lemma monotonicity f lambda}, this proves the asserted coercivity property of the function $\max \Theta(\cdot)$. The proof of the second inequality is analogous and is omitted. 
\end{proof}

With the aid of Proposition \ref{prop Theta}, we proceed to show the following crucial result. 

\begin{prop}\label{prop theta image}\ 
The set $\D{Im}(\Theta):=\bigcup\limits_{\lambda\geqslant \lambda_0} \Theta(\lambda)$ is a closed and unbounded subset of $\R$.  
\end{prop}
\begin{proof}
The fact that  $\D{Im}(\Theta)$ is unbounded is a direct consequence of Proposition \ref{prop Theta}-(v). Let us show it is closed. Pick an accumulation point $\theta$ of $\D{Im}(\Theta)$ and a sequence $(\theta_n)_n\subset \D{Im}(\Theta)$ such that $\theta_n \to \theta$. Let $(\lambda_n, f_n)\in [\lambda_0, +\infty)\times\Sol_{\lambda_n}$ such that $\theta_n:=\EE[f_n(0,\cdot)]$.\medskip

\noindent{\bf Claim: $\mathbf (\lambda_n)_n$ is bounded.}\smallskip

Up to extracting a subsequence, we can assume that the $\theta_n$ have constant sign.\smallskip

\noindent\underline{Case $\theta_n\geqslant 0$ for all $n\in\N$.}\smallskip\\
For every $n\in\N$, let us set $C_n(\omega):=\{x\in\R\,:\, f_n(x,\omega)\geqslant \theta_n\}$. Then $C_n(\omega)$ is an almost surely nonempty, closed random stationary set, in particular 
\begin{equation}\label{eq random set}
C_n(\omega)\cap(-\infty,-k)\not=\emptyset\quad\hbox{and}\quad C_n(\omega)\cap(k,+\infty)\not=\emptyset \quad\hbox{for all  $k\in\N$}\qquad \hbox{almost surely,}
\end{equation}
see Propositions 3.2. and 3.5 in \cite{DS09}, for instance. 

If $f_n(\cdot,\omega)\geqslant 0$ in $\R$ almost surely, then, according to Lemma \ref{lem:stat}, for every fixed $\omega$ in a set of probability 1, there exists a local minimum point $y\in\R$ of $f(\cdot,\omega)$ in $\R$ with $0\leqslant f(y,\omega)\leqslant \theta_n$. Hence 
\[
\lambda_n 
=
a(y,\omega)f'_n(y,\omega)+H(y,f_n(y,\omega),\omega)
=
H(y,f_n(y,\omega),\omega)
\leqslant
\alpha_1(|f_n(y,\omega)|+1)
\leqslant
\alpha_1 (\theta_n+1). 
\]
If, on the other hand, $f_n(\cdot,\omega)$ changes sign almost surely, then, according to \eqref{eq random set}, for every fixed $\omega$ in a set of probability 1, there exists a point $z\in\R$ such that $f_n(z,\omega)=0$ and $f'_n(z,\omega)\leqslant 0$. Hence 
\[
\lambda_n 
=
a(z,\omega)f'_n(z,\omega)+H(z,f_n(z,\omega),\omega)
\leqslant 
H(z,0,\omega)
\leqslant
\alpha_1. 
\]
In either case, we get the claim.\medskip\\
\noindent\underline{Case $\theta_n< 0$ for all $n\in\N$.}\smallskip\\
The function $\hat f_n(x,\omega):=-f_n(-x,\omega)$ is a solution of \eqref{eq ODE lambda} with 
$\lambda:=\lambda_n$, $\hat a(x,\omega):=a(-x,\omega)$ and $\hat H(x,p,\omega):=H(-x,-p,\omega)$. The assertion follows by applying the previous step to the function $\hat f_n$.\smallskip
 
We have thus shown that the sequence $(\lambda_n)_n$ is bounded. Up to extracting a subsequence, we can assume that $\lambda_n\to \lambda$ in $[\lambda_0,+\infty)$. We conclude that $\theta\in\Theta(\lambda)$ by the upper semicontinuity of the set-valued map $\Theta(\cdot)$, see Proposition \ref{prop Theta}-(iii). 
\end{proof}

Last, we show the following result, which is the final step to establish homogenization of equation \eqref{eq:introHJ}.

\begin{prop}\label{prop D} Let us set $D:=\{\theta\in\R\,:\,\HV^L(H)(\theta)=\HV^U(H)(\theta)\}$.  Then $D=\R$.  
\end{prop}

\begin{proof}
In view of Proposition \ref{prop consequence existence corrector}, we already know that $\D{Im}(\Theta)\subseteq D$. 
The set $A:=\R\setminus\D{Im}(\Theta)$ is an open set in view of Proposition \ref{prop theta image}. Let us assume that $A\not=\emptyset$ and let us denote by $I$ a connected component of $A$. Due to Proposition \ref{prop Theta}-(v) and the fact that $\R$ is locally connected, we infer that $I$ is a bounded open interval of the form $(\theta_1,\theta_2)$ with $\theta_i\in\Theta(\lambda_i)$ for $i\in\{1,2\}$ and $\lambda_1,\lambda_2\in [\lambda_0,+\infty)$. The fact that $I\cap \D{Im}(\Theta)=\emptyset$ implies, in view of Proposition \ref{prop consequence existence corrector}, that
\begin{equation}\label{eq null intersection}
\{f\in\Sol_\mu\,:\, \EE[f(0,\cdot)]=\theta\ \hbox{for some $\mu\geqslant \lambda_0$}\}=\emptyset
\qquad
\hbox{for every fixed $\theta\in I$.}
\end{equation}

\medskip

\noindent{\bf Claim: $\mathbf \lambda_1=\lambda_2$.}\smallskip

Let us assume the claim false. Let $f_i\in\Sol_{\lambda_i}$ such that $\EE[f(0,\cdot)]=\lambda_i$ for $i\in\{1,2\}$ with $\lambda_1\not=\lambda_2$. According to Lemma \ref{lem:ordqu} and the fact that $\theta_1<\theta_2$, we must have, almost surely, $f_1(\cdot,\omega)<f_2(\cdot,\omega)$ in $\R$. Pick a constant $\mu$ strictly in between $\lambda_1$ and $\lambda_2$. Let us denote by $\Solfrak(\mu; f_1,f_2)(\omega)$ the set of functions $f\in C^1(\R)$ which solve equation \eqref{eq ODE lambda} with $\mu$ in place of $\lambda$ and satisfy $f_1(\cdot,\omega)< f < f_2(\cdot,\omega)$ in $\R$, for every fixed $\omega\in\Omega$.  Such set is almost surely nonempty. This follows by applying Lemma \ref{lem:inbetw} with $M(\cdot)=f_2(\cdot,\omega)$, $m(\cdot):=f_1(\cdot,\omega)$ and $G:=H-\mu$. Let us set 
\[
\overline g(x,\omega):=\sup \left\{f(x,\omega)\,:\, f\in\Solfrak(\mu; f_1,f_2)(\omega)\right\},\qquad (x,\omega)\in\R\times\Omega. 
\]
By arguing as in the proofs of Lemma \ref{lemma measurable minimal sol} and Theorem \ref{teo lower bound}, we derive that $g$ is jointly measurable and stationary, and satisfies $\overline g(\cdot,\omega)\in \Solfrak(\mu; f_1,f_2)(\omega)$ almost surely. This means that $\overline g\in S_\mu$ with 
$\theta:=\EE[\overline g(0,\cdot)]\in (\theta_1,\theta_2)\cap\D{Im}(\Theta)$, in contradiction with \eqref{eq null intersection}. 

We have thus proved that $\lambda_1=\lambda_2=:\lambda$. By exploiting \eqref{eq null intersection} again, we infer, in view of Theorem \ref{teo lower bound}, that 
\begin{equation*}\label{eq2 flat part}
\lambda\leqslant \HV^L(H)(\theta)\leqslant \HV^U(H)(\theta)\leqslant \lambda
\qquad
\hbox{for every $\theta\in I$,}
\end{equation*}
i.e., $I\subset D$. 
\end{proof}

\begin{remark}\label{oss flat part}
It is worth pointing out that what we have shown above is that any connected component $I$ of $\R\setminus\D{Im}(\Theta)$ corresponds to a flat part 
of the effective Hamiltonian $\HV(H)$. 
\end{remark}

From the information gathered, we can now prove the homogenization result stated in Theorem \ref{thm:genhom}. 

\begin{proof}[Proof of Theorem \ref{thm:genhom}] In view of Proposition \ref{prop reduction}, we can assume that $H\in\Ham$ for some constants $\alpha_0,\alpha_1>0$ and $\gamma>2$. 
From Proposition \ref{prop D} we derive that, for every fixed $\theta$ we have $\HV^L(\theta)=\HV^U(\theta)$. 
We denote this common value by $\HV(H)(\theta)$. 
According to Proposition \ref{appB prop HF}, this defines a function $\HV(H):\R\to [\lambda_0,+\infty)$ which is superlinear and locally Lipschitz. Homogenization of equation \eqref{eq:introHJ} follows in view of \cite[Lemma 4.1]{DK17}.
\end{proof}

%

\appendix

\section{ODE results}\label{app:ODE}

We state here for the reader's convenience a series of results that we repeatedly used throughout the paper and which hold true for rather general continuous and coercive Hamiltonians. They had been obtained by the author for the current research and had been already shown their usefulness in the work  \cite{DKY23}. We refer the reader to \cite[Appendix A]{DKY23} for the proofs.  

We briefly recall the setting under which such results hold. We will work with a general probability space 
$(\Omega,\F, \P)$, where $\P$ and $\F$ denote the probability measure on $\Omega$  and
the $\sigma$--algebra of $\P$--measurable subsets of $\Omega$, respectively. 
We will assume that  $\P$ is invariant under the action of a one-parameter group $(\tau_x)_{x\in\R}$ of transformations $\tau_x:\Omega\to\Omega$ and that the action of $(\tau_x)_{x\in\R}$ is ergodic. No topological or completeness assumptions are made on the probability space.\smallskip

Let $a:\R\times\Omega\to [0,1]$  and $G:\R\times\R\times\Omega\to[m_0,+\infty)$ be stationary functions with respect to the shifts in the $x$-variable, were $m_0$ is real constant such that $m_0=\min_{\R\times\R} G(\cdot,\cdot,\omega)$ almost surely. 
We will assume that $a$ is continuous in the first variable and $G$ is continuous in the first two variables, for any fixed $\omega\in\Omega$.


We introduce the following conditions on $G$ and specify in each statement which of them are needed for that result:
\begin{itemize}
	\item[(G1)] \quad there exist two coercive functions
	$\alpha_G,\beta_G:[0,+\infty)\to\R$ such that
	\[
	\alpha_G\left(|p|\right)\leqslant G(p,x,\omega)\leqslant \beta_G\left(|p|\right)\quad\hbox{for every $(p,x,\omega)\in\R\times\R\times\Omega$;}
	\]
	\item[(G2)] \quad for every fixed $R>0$, there exists a constant $C_R>0$ such that 
	\[
	|G(p,x,\omega)-G(q,x,\omega)|\leqslant C_R |p-q|\quad\hbox{for every $p,q\in [-R,R]$ and $(x,\omega)\in\R\times\Omega$}.
	\] 
\end{itemize}

For our first statement, we will need the following notation: for each $\lambda\geqslant 0$,
\begin{align}\label{eq:pla} 
	p_\lambda^-:=\inf_{x\in\R}\inf\{p\in\R:\ G(p,x,\omega)\leqslant \lambda\} \quad\text{and}\quad p^+_\lambda:=\sup_{x\in\R}\sup\{p\in\R:\ G(p,x,\omega)\leqslant \lambda\}.
\end{align}
By the ergodicity assumption and (G1), the quantities $p^\pm_\lambda$ are a.s.\ constants. The functions $\lambda\mapsto p^-_\lambda$ and $\lambda\mapsto p^+_\lambda$ are, respectively, non-increasing and non-decreasing (and, in general, not continuous). Furthermore, $p^-_\lambda\to-\infty$ and $p^+_\lambda\to +\infty$ as $\lambda\to+\infty$.

\begin{lemma}\label{lem:ubounds}
	Assume that $G$ satisfies {\em (G1)}. Take any $\lambda>m_0$. Let $f(x,\omega)$ be a stationary function such that, for all $\omega\in\Omega$, $f(\,\cdot\,,\omega)\in\CC^1(\R)$, and
	\begin{equation*}\label{eq:lesslam}
		a(x,\omega)f'(x,\omega)+G(f(x,\omega),x,\omega)< \lambda\quad\forall x\in\R. 
	\end{equation*}
	Then, on a set $\Omega_f$ of probability 1, $f(x,\omega)\in (p^-_\lambda,p^+_\lambda)$ for all $x\in\R$, where $p^\pm_\lambda$ are defined in \eqref{eq:pla}.
\end{lemma}

\begin{lemma}\label{lem:stat}
	Let $f(x,\omega)$ be a stationary function such that $f(\,\cdot\,,\omega) \in\CC(\R)$ for every $\omega\in\Omega$. Then, we have the following dichotomy:
	\begin{itemize}
		\item [(i)] $\P(f(x,\omega) = c\ \forall x\in\R) = 1$ for some constant $c\in\R$;
	    \item [(ii)] for $\P$-a.e.\ $\omega$, $f(\,\cdot\,,\omega)$ has infinitely many local maxima and minima.
	\end{itemize}
\end{lemma}

As a consequence of Lemma \ref{lem:ubounds}, we infer
\begin{cor}\label{cor:lelam}
	Assume that $G$ satisfies {\em (G1)}. Take any $\lambda\geqslant m_0$. Let $f(x,\omega)$ be a stationary function such that, for all $\omega\in\Omega$, $f(\,\cdot\,,\omega)\in\CC^1(\R)$, and
	\[ a(x,\omega)f'(x,\omega)+G(f(x,\omega),x,\omega)\leqslant \lambda\quad\forall x\in\R. \]
	Then, on a set $\Omega_f$ of probability 1,
	$\displaystyle f(x,\omega)\in[\sup_{\mu>\lambda}p^-_\mu,\inf_{\mu>\lambda}p^+_\mu]$ for all $x\in\R$.
\end{cor}

The next result generalizes the fact that, under suitable conditions, two distinct solutions of an ODE do not touch each other.

\begin{lemma}\label{lem:order}
	Assume that $H$ satisfies {\em (H2)}, and $a(x,\omega) > 0$ for all $(x,\omega)\in\R\times\Omega$. Let $f_1(x,\omega)$ and $f_2(x,\omega)$ be stationary processes such that, for all $\omega\in\Omega$, $f_1(\,\cdot\,,\omega),f_2(\,\cdot\,,\omega)\in\CC^1(\R)\cap \CC_b(\R)$, and
	\begin{equation}\label{eq:ord1}
		a(x,\omega)f_1'(x,\omega)+H(f_1(x,\omega),x,\omega) \leqslant a(x,\omega)f_2'(x,\omega)+H(f_2(x,\omega),x,\omega)\quad \forall x\in\R.
	\end{equation}
	Then, one of the following events has probability $1$:
	\begin{align*}
		\Omega_0\, &= \{\omega\in\Omega:\, (f_1-f_2)(x,\omega) = 0\ \text{for all}\ x\in\R \};\\
		\Omega_- &= \{\omega\in\Omega:\, (f_1-f_2)(x,\omega) < 0\ \text{for all}\ x\in\R \};\\
		\Omega_+ &= \{\omega\in\Omega:\, (f_1-f_2)(x,\omega) > 0\ \text{for all}\ x\in\R \}.
	\end{align*}
\end{lemma}

The next result shows that two distinct solutions of an ODE do not touch each other and gives a quantitative estimate of the distance between them.
%
%


\begin{lemma}\label{lem:ordqu}
	Assume that $G$ satisfies {\em (G1)} and {\em (G2)}. Take any $\lambda_1,\lambda_2\in\R$ such that $m_0\leqslant\lambda_1 < \lambda_2$. Let $f_1(x,\omega)$ and $f_2(x,\omega)$ be stationary functions such that, for all $i\in\{1,2\}$ and $\omega\in\Omega$, $f_i(\,\cdot\,,\omega)\in\CC^1(\R)$, and
	\begin{equation*}
		a(x,\omega)f_i'(x,\omega) + G(f_i(x,\omega),x,\omega) = \lambda_i\quad\forall x\in\R.
	\end{equation*}
	Then, there is a constant $\delta>0$, which depends only on $\lambda_2$ and $G$, such that
	\[\P((f_1-f_2)(x,\omega)>\delta(\lambda_2-\lambda_1)\ \forall x\in\R)=1\ \ \text{or}\ \ \P((f_2-f_1)(x,\omega)>\delta(\lambda_2-\lambda_1)\ \forall x\in\R)= 1.\]
\end{lemma}

The next statement is deterministic. One should think of $a(\,\cdot\,)$ and $G(\,\cdot\,,\,\cdot\,)$ as $a(\,\cdot\,,\omega)$ and $G(\,\cdot\,,\,\cdot\,,\omega)$ with $\omega$ fixed.  
It states that we can always insert a global $\CC^1$ solution between two bounded
strict sub and supersolutions which do not intersect. The result is standard, a proof can be found \cite[Lemma A.8]{DKY23}

\begin{lemma}\label{lem:inbetw}
	Assume that $G(p,x)$ satisfies {\em (G1)} and {\em (G2)}, $a(x)>0$ for all $x\in\R$, there exist bounded functions $m,M\in\CC^1(\R)$ such that $m(x)<M(x)$ for all $x\in\R$, and either one of the following inequality holds:
\begin{itemize}
\item[(i)]\quad $a(x)m'(x)+G(m(x),x) < 0 < a(x)M'(x)+G(M(x),x)\qquad\forall x\in\R$;\medskip
\item[(ii)]\quad $a(x)m'(x)+G(m(x),x) > 0 > a(x)M'(x)+G(M(x),x)\qquad\forall x\in\R$.\smallskip
\end{itemize}
Then, there exists a function $ f\in\CC^1(\R)$ that solves the equation
	\begin{align*}
		a(x)f'(x)+G(f(x),x)=0\quad &\forall x\in \R, 
		\shortintertext{and satisfies}
		m(x) < f(x) < M(x)\quad&\forall x\in\R.\\
	\end{align*}
\end{lemma}

We end this section with the following useful stability result, see \cite[Lemma A.9]{DKY23} for a proof.  

\begin{lemma}\label{appendix A lemma lattice}
Assume $G(p,x)$ satisfies (G2) and $a(x)>0$ for all $x\in\R$. Let $\Sol_\lambda$ be a nonempty family of solutions to
\begin{equation}\label{appendix A eq ODE}
	a(x)u'(x) + G(u(x),x) = \lambda,\quad x\in\R.
\end{equation}
If $\Sol_\lambda$ is a compact subset of $\D{C}(\R)$, then the functions 
\[
\ul u (x):=\inf_{u\in\Sol_\lambda} u(x),
\quad
\ol u(x):=\sup_{u\in\Sol_\lambda} u(x),
\qquad
x\in\R,
\] 
are in $\Sol_\lambda$. 
\end{lemma}

\section{PDE results}\label{app:PDE}
In this appendix, we collect some known PDE results that we need in the paper. Throughout this section, we will denote by 
$\D{UC}(X)$, $\D{LSC}(X)$ and $\D{USC}(X)$ the space of uniformly continuous, lower semicontinuous and upper semicontinuous real functions on a metric space $X$, respectively. 
We will denote by $H$ a continuous function defined on $\R\times\R$. If not otherwise stated, we shall assume that 
$H$ belongs to the class $\Ham$  introduced in Definition \ref{def:Ham}, for some constants $\alpha_0,\alpha_1>0$ and $\gamma>1$. 

We will assume that $a:\R\to [0,1]$ is a function satisfying the following assumption, for some constant $\kappa  > 0$: 
\begin{itemize}
\item[(A)]  $\sqrt{a}:\R\to [0,1]$\ is $\kappa $--Lipschitz continuous.
\end{itemize}
Note that (A) implies that $a$ is $2\kappa$--Lipschitz in $\R$.\smallskip 

%
%

\subsection{Stationary equations} Let us consider a stationary viscous HJ equation of the form
\begin{equation}\label{eq PDE}
	a(x)u''(x)+H(x,u')=\lambda\qquad \hbox{ in $\R$,}
		\tag{SHJ}
\end{equation}
where $\lambda\in\R$, the nonlinearity $H$ belongs to $\Ham$, and  $a:\R\to [0,1]$ satisfies condition (A). 
The following holds. 
\begin{prop}\label{prop regularity solutions}
Let $u\in\CC(\R)$ be a viscosity solution of \eqref{eq PDE}. Let us assume that $H\in\Ham$ for some constants $\alpha_0,\alpha_1>0$ and $\gamma>1$. Then 
\[
|u(x)-u(y)| \leqslant K |x-y|
\qquad\hbox{for all $x,y\in\R$,}
\]
where $K>0$ is given explicitly by 
\begin{equation}\label{eq Lipschitz bound}
K:=C\left( 
		\left(
			\kappa  
			\dfrac{\sqrt{1+\alpha_1+|\lambda|}}{\alpha_0}
		\right)^{\frac{2}{\gamma-1}} + 
	\left(
	\dfrac{1+\lambda\alpha_0}{\alpha^2_0}
	\right)^{\frac1\gamma}
\right)
\end{equation}
with $C>0$ depending only on $\gamma$.
Furthermore, $u$ is of class $C^2$ (and hence a pointwise solution of \eqref{eq PDE}) in every open interval $I$ where $a(\cdot)$ is strictly positive. 
\end{prop}

\begin{proof}
The Lipschitz character of $u$ is direct consequence of \cite[Theorem 3.1]{AT}, to which we refer for a proof.  
Let us now assume that $a(\cdot)$ is strictly positive on some open interval $I$. Without loss of generality, we can assume that $I$ is bounded and $\inf_I a>0$. From the Lipschitz character of $u$ we infer that $-C\leq u''\leq C$ in $I$ in the viscosity sense 
for some constant $C>0$, or, equivalently, in the distributional sense, in view of  \cite{Is95}. Hence, 
$u'' \in L^\infty(I)$. The elliptic regularity theory, see \cite[Corollary 9.18]{GilTru01}, ensures that $u\in W^{2,p}(I)$ for any $p>1$
and, hence, $u\in \D{C}^{1,\sigma}(I)$ for any $0<\sigma<1$. Since $u$ is a viscosity solution to \eqref{eq PDE} in $I$, 
by Schauder theory \cite[Theorem 5.20]{HL97}, we conclude that $u\in \D{C}^{2,\sigma}(I)$ for any $0<\sigma<1$.
\end{proof}

We shall also need the following H\"older estimate for supersolutions of \eqref{eq PDE}.

\begin{prop}\label{prop Holder estimate}
Let us assume that $H\in\Ham$ with $\gamma>2$. Let $u\in\CC(\R)$ be a supersolution of \eqref{eq PDE} for some 
$\lambda\in\R$.  Then 
\begin{equation*}
|u(x)-u(y)| \leqslant K |x-y|^{\frac{\gamma-2}{\gamma-1}}\qquad\hbox{for all $x,y\in\R$,}
\end{equation*}
where $K>0$ is given explicitly by 
\begin{equation}\label{eq Holder bound}
K:=C\left( 
		\left(
			\dfrac{1}{\alpha_0}
		\right)^{\frac{1}{\gamma-1}} + 
	\left(
	\dfrac{1+\lambda\alpha_0}{\alpha^2_0}
	\right)^{\frac1\gamma}
\right)
\end{equation}
with $C>0$ depending only on $\gamma$.
\end{prop}

\begin{proof}
The function $v(x):=-u(x)$ is a viscosity subsolution of \eqref{eq PDE} with $-a(\cdot)$ in place of $a(\cdot)$ and $\check H(x,p):=H(x,-p)$ in place of $H$. By the fact that $a\leqslant 1$ and $H\in\Ham$, we derive that $v$ satisfies the following inequality in the viscosity sense:
\[
-|v''|+\alpha_0|v'|^\gamma\leqslant \lambda+\dfrac{1}{\alpha_0}\quad\hbox{in $\R$}.
\]
The conclusion follows by applying \cite[Lemma 3.2]{AT}.
\end{proof}

\subsection{Parabolic equations}
Let us consider a parabolic PDE of the form
\begin{equation}\label{appB eq parabolic HJ}
\partial_{t }u=a(x) \partial^2_{xx} u +H(x,\partial_x u), \quad(t,x)\in (0,+\infty)\times\R.
\tag{EHJ}
\end{equation}

We start with a comparison principle stated in a form which is the one we need in the paper.

\begin{prop}\label{appB prop comparison}
Suppose $a$ satisfies (A) and  $H\in\D{UC}\left(B_r\times\R\right)$ for every $r>0$. Let $v\in\D{USC}([0,T]\times\R)$ and $w\in\D{LSC}([0,T]\times\R)$ be, respectively, a sub and a supersolution of \eqref{appB eq parabolic HJ} in $(0,T)\times \R$ such that 
\begin{equation}\label{hyp 2}
\limsup_{|x|\to +\infty}\ \sup_{t\in [0,T]}\frac{v(t,x)-\theta x}{1+|x|}\leqslant 0 
\leqslant 
\liminf_{|x|\to +\infty}\ \sup_{t\in [0,T]}\frac{w(t,x)-\theta x}{1+|x|}
\end{equation}
for some $\theta\in\R$. Let us furthermore assume that either $\partial_x v$ or $\partial_x w$ belongs to $L^\infty\left((0,T)\times \R\right)$. Then, 
\[
v(t,x)-w(t,x)\leqslant \sup_{\R}\big(v(0,\,\cdot\,) - w(0,\,\cdot\,)\big)\quad\hbox{for every  $(t,x)\in (0,T)\times \R$.} 
\]
\end{prop}

\begin{proof}
The functions $\tilde v(t,x):=v(t,x)-\theta x$ and $\tilde w(t,x):=w(t,x)-\theta x$ are, respectively, a subsolution and a supersolution of \eqref{appB eq parabolic HJ} in $(0,T)\times \R$ with $H(\,\cdot\,,\theta +\,\cdot\,)$ in place of $H$. The assertion follows by applying \cite[Proposition 1.4]{D19} to $\tilde v$ and $\tilde w$.
\end{proof}

Let us now assume that $H$ belongs to the class $\Ham$ for some fixed constants $\alpha_0,\alpha_1>0$ and $\gamma>1$.
%
%
The following holds. 

\begin{theorem}\label{appB teo well posed}
Suppose $a$ satisfies (A) and $H\in\Ham$. Then, for every $g\in\D{UC}(\R)$, there exists a unique function $u\in \D{UC}(\ccyl)$ that solves the equation \eqref{appB eq parabolic HJ}
subject to the initial condition $u(0,\,\cdot\,)=g$ on $\R$. If $g\in W^{2,\infty}(\R)$, then $u$ is Lipschitz continuous in $\ccyl$ and satisfies
\begin{equation*}
\|\partial_t u\|_{L^\infty(\ccyl)}\leqslant K 
\quad\text{and}\quad
\|\partial_x u\|_{L^\infty(\ccyl)}\leqslant K 
\end{equation*}
for some constant $K $ that depends only on $\|g'\|_{L^\infty(\R)}$, $\|g''\|_{L^\infty(\R)}$, $\kappa , \alpha_0,\alpha_1$ and $\gamma$. Furthermore, the dependence of $K $ on $\|g'\|_{L^\infty(\R)}$ and $\|g''\|_{L^\infty(\R)}$ is continuous. 
\end{theorem}

\begin{proof}
A proof of this result when the initial datum $g$ is furthermore assumed to be bounded is given in \cite[Theorem 3.2]{D19}, see also \cite[Proposition 3.5]{AT}. 
This is enough, since we can always reduce to this case by possibly picking a function $\tilde g\in W^{3,\infty}(\R)\cap\CC^\infty(\R)$ such that $\|g-\tilde g\|_{L^\infty(\R)}<1$ (for instance, by mollification) and by considering equation \eqref{appB eq parabolic HJ} with
$\tilde H(x,p):=a(x)(\tilde g)'' +H(x,p + (\tilde g)')$ in place of $H$, and initial datum $g-\tilde g$. 
\end{proof}

For every fixed $\theta\in\R$, we will denote by $u_\theta$ the unique solution of \eqref{appB eq parabolic HJ} in $\D{UC}(\ccyl)$ satisfying  $u_\theta(0,x)=\theta x$ for all $x\in\R$. 
Theorem \ref{appB teo well posed} tells us that $u_\theta$ is Lipschitz in $\ccyl$ and that its Lipschitz constant depends continuously on $\theta$. 
%
%
Let us define 
\begin{eqnarray}\label{appB eq infsup}
	\HF^L(H) (\theta):=\liminf_{t\to +\infty}\ \frac{u_\theta(t,0)}{t}\quad\text{and}
	\quad
	\HF^U(H) (\theta):=\limsup_{t\to +\infty}\ \frac{u_\theta(t,0)}{t}.
\end{eqnarray}
By definition, we have \ $\HF^L(H) (\theta)\leqslant \HF^U(H) (\theta)$\ for all $\theta\in\R$. 
Furthermore, the following holds, see \cite[Proposition B.3]{DKY23} for a proof.

\begin{prop}\label{appB prop HF}
Suppose $a$ satisfies (A) and $H\in\Ham$. Then the functions $\HF^L(H)$ and $\HF^U(H)$ satisfy (H1) and are locally Lipschitz on $\R$. 
\end{prop}

\bibliography{viscousHJ}
\bibliographystyle{siam}
\end{document}